\documentclass[12pt,reqno]{amsart}
\usepackage{amsmath} 

\usepackage{amsthm}
\usepackage{amssymb}
\usepackage{graphics}
\usepackage{latexsym}
\usepackage{comment}

\numberwithin{equation}{section}
\newtheorem{thm}{Theorem}[section]
\newtheorem{prop}[thm]{Proposition}
\newtheorem{lem}[thm]{Lemma}
\newtheorem{cor}[thm]{Corollary}

\theoremstyle{remark}
\newtheorem{rem}{Remark}[section]
\newtheorem{defn}{Definition}

\newcommand{\BBB}{\mathbb}
\newcommand{\R}{{\BBB R}}
\newcommand{\Z}{{\BBB Z}}
\newcommand{\T}{{\BBB T}}

\newcommand{\N}{{\BBB N}}

\newcommand{\HT}{{\mathcal H}}%
\newcommand{\LR}[1]{{\langle {#1} \rangle }}

\newcommand{\lec}{{\ \lesssim \ }}

\newcommand{\al}{\alpha}

\newcommand{\ga}{\gamma}

\newcommand{\vp}{\varphi}

\newcommand{\e}{\varepsilon}

\newcommand{\la}{\lambda}
\newcommand{\La}{\Lambda}

\newcommand{\de}{\delta}

\newcommand{\supp}{\operatorname{supp}}

\newcommand{\I}{\infty}

\newcommand{\sgn}{\operatorname{sgn}}

\newcommand{\EQS}[1]{\begin{align} #1 \end{align}}
\newcommand{\EQQS}[1]{\begin{align*} #1 \end{align*}}
\newcommand{\EQQ}[1]{\begin{equation*} \begin{split} #1
 \end{split} \end{equation*}}

\newcommand{\F}{\mathcal{F}}

\newcommand{\ti}{\widetilde}
\newcommand{\ha}{\widehat}

\newcommand{\ds}{\partial_{x}}%
\newcommand{\dt}{\partial_{t}}%
\title[L.W.P. for 4thBO]
{Local well-posedness for fourth order Benjamin-Ono type equations}

\author[T. Tanaka]{Tomoyuki Tanaka}
\address[T. Tanaka]{Graduate School of Mathematics, Nagoya University,
Chikusa-ku, Nagoya, 464-8602, Japan}
\email[T. Tanaka]{d18003s@math.nagoya-u.ac.jp}

\begin{document}
\maketitle

\begin{abstract}
We continue to study the local well-posedness for higher order Benjamin-Ono type equations, especially fourth order equations.
The proof is based on the energy methods with correction terms.
Although one of correction terms can eliminate the highest order derivative loss in the energy inequality, it may yield a lower order derivative loss than the worst term.
In order to cancel this derivative loss, we define correction terms inductively.
\end{abstract}
\maketitle
\setcounter{page}{001}

\section{Introduction}

We consider the Cauthy problem of the following fourth order Benjamin-Ono type equations:
\EQS{
\label{BO1}
&\dt u=\ds K(u),\\
\label{BO2}
&u(0,x)=\vp(x),
}
where $t\in\R$, $x\in \R$ or $\T(=:\R/2\pi\Z)$, $u=u(t,x),\vp=\vp(x)\in\R$,
\EQS{\label{K}
\begin{aligned}
K(u)
&:=\HT\ds^{3}u+c_{1}u\ds^{2}u+c_{2}(\ds u)^{2}+c_{3}(\HT\ds u)^{2}+c_{4}\HT(u\HT\ds^{2}u)\\
&\quad+c_{5}\HT(u^{2}\ds u)+c_{6}u\HT(u\ds u)+c_{7}u^{2}\HT\ds u-u^{4}
\end{aligned}
}
and $c_{j}\in\R$ for $j=1,\dots,7$.
$\HT$ is the Hilbert transform defined by
\[\ha{\HT f}(0)=0\quad{\rm and}\quad\ha{\HT f}(\xi)=-i\sgn(\xi)\hat{f}(\xi)\]
for $\xi\in\R\backslash\{0\}$ or $\Z\backslash\{0\}$
where $\hat{f}$ is the Fourier transform of $f$: $\hat{f}(\xi)=\F f(\xi)=(2\pi)^{-1}\int f(x)e^{-ix\xi}dx$.
The well-known Benjamin-Ono equation
\EQS{
\label{2BO}
\dt u+\HT\ds^{2}u+2u\ds u=0
}
describes the behavior of long internal waves in deep stratified fluids.
The equation \eqref{2BO} also has infinitely many conservation laws,
which generates a hierarchy of Hamiltonian equations of order $j$.
The equation \eqref{BO1} with $c_{1}=3$, $-c_{2}=c_{5}=c_{6}=c_{7}=-2$ and $c_{3}=c_{4}=-1$ is integrable
and the third equation in the Benjamin-Ono hierarchy \cite{Matsuno}.

There are a lot of literature on the Cauchy problem on \eqref{2BO}.
On the real line case,
Ionescu-Kenig \cite{IoKe07} showed the local well-posedness in $H^{s}(\R)$ for $s\ge0$
(see also \cite{Molinet12} for another proof
and \cite{IoKe072} for the local well-posedness with small complex valued data).
On the periodic case, Molinet \cite{Molinet08,Molinet09} showed the local well-posedness
in $H^{s}(\T)$ for $s\ge0$ and that this result was sharp.
See \cite{ABFS,BuPl,Iorio86,KeKe,KoTz,Ponce,Tao} for former results.

In \cite{Tanaka}, we studied the local well-posedness for the equation
\EQS{
\label{3BO}
\dt u=\ds(\ds^{2}u+d_{1}u\HT\ds u+d_{2}\HT(u\ds u)-u^{3}),\quad x\in\T,
}
where $d_{1},d_{2}\in\R$.
The equation \eqref{3BO} with $d_{1}=d_{2}=3/2$ is integrable and the second equation in the Benjamin-Ono hierarchy.
The local well-posedness for \eqref{3BO} is based on the energy method with a correction term.
Namely, we employ the energy method to
\[E_{*}(u)
:=\|u\|_{L^{2}}^{2}+\|D^{s}u\|_{L^{2}}^{2}+a_{s}\|u\|_{L^{2}}^{4s+2}
  +b_{s}\int u(\HT D^{s}u)D^{s-2}\ds udx\]
(see Definition 2 in \cite{Tanaka}) in order to eliminate the first order derivative loss.
In fact, we have the second order derivative loss resulting from nonlinear terms in the energy inequality, but it can be reduced to the first order derivative loss because of the symmetry (see Lemma 2.6 in \cite{Tanaka}).
For related results such as the local well-posedness on the real line, see \cite{FeHa96,Feng97,LPP11,MP12}.

On the other hand, as far as we know, there are no well-posedness results for \eqref{BO1} either on the real line or on the torus.
In particular, some of nonlinear terms in \eqref{BO1} have three derivatives, which implies that the local well-posedness for \eqref{BO1} is far from trivial.
The main result is the following:

\begin{thm}\label{main}
We write $M=\R$ or $\T$.
Let $s\ge s_{0}>7/2$.
For any $\vp\in H^{s}(M)$,
there exist $T=T(\|\vp\|_{H^{s_{0}}})>0$ and the unique solution $u\in C([-T,T];H^{s}(M))$
to the IVP \eqref{BO1}--\eqref{BO2} on $[-T,T]$.
Moreover, for any $R>0$,
the solution map $\vp\mapsto u(t)$ is continuous
from the ball $\{\vp\in H^{s}(M);\|\vp\|_{H^{s}}\le R\}$ to $C([-T,T];H^{s}(M))$.
\end{thm}

Now we mention the idea of the proof of Theorem \ref{main}.
We may have the third order derivative loss since nonlinear terms in \eqref{BO1} have three derivatives at most.
By the symmetry, it can be reduced to the second order derivative loss (see Lemma \ref{lem4.5}).
Our proof is based on the energy method, and the standard energy estimate gives only the following:
\EQS{\label{deri.los.}
\frac{d}{dt}\|D^{s}u(t)\|_{L^{2}}^{2}
\le C(1+\|u\|_{H^{s_{0}}})^{3}\|u(t)\|_{H^{s}}^{2}+|L_{1}(u)|+|L_{2}(u)|+|L_{3}(u)|,
}
where $s_{0}>7/2$, $D=\F^{-1}|\xi|\F$ and
\EQQS{
&L_{1}(u):=\la_{1}(s)\int\ds u(D^{s}\ds u)^{2}dx,
\quad L_{2}(u):=\la_{2}(s)\int(\HT\ds^{2}u)(\HT D^{s}\ds u)D^{s}udx,\\
&L_{3}(u):=\la_{3}(s)\int u\ds u (\HT D^{s}\ds u)D^{s}udx
}
(see Definition \ref{corr.terms.} for definitions of $\la_{j}(s)$).
Here, we note that $L_{1}(u)$ is the second order derivative loss, and $L_{2}(u)$ and $L_{3}(u)$ are the first order derivative losses.
We need to to handle $L_{j}(u)$ for $j=1,2,3$ by $\|u\|_{H^{s}}$ if we use the standard argument.
However, it is impossible to do that.
In order to overcome this difficulty, we modify the energy by adding correction terms, following the idea from Kwon \cite{Kwon} who studied the local well-posedness for the fifth order KdV equation (see also Segata \cite{Segata}, Kenig-Pilod \cite{KePi16} and Tsugawa \cite{Tsugawa17}).
Namely, we consider
\EQQS{
E_{s}(u):=\frac{1}{2}\|u\|_{L^{2}}^{2}(1+C_{s}\|u\|_{L^{2}}^{2}+C_{s}\|u\|_{L^{2}}^{4s})
            +\frac{1}{2}\|D^{s}u\|_{L^{2}}^{2}+\sum_{j=1}^{3}M_{s}^{(j)}(u),
}
with
\EQQS{
&M_{s}^{(1)}(u):=\frac{\la_{1}(s)}{4}\int u(\HT D^{s}u)\HT D^{s-1}udx,\\
&M_{s}^{(2)}(u):=\frac{\la_{2}(s)}{4}\int (\HT\ds u)(D^{s-1}u)^{2}dx,\\
&M_{s}^{(3)}(u):=\frac{\la_{1}(s)\la_{4}(s)+4\la_{3}(s)}{32}\int u^{2}(D^{s-1}u)^{2}dx
}
(see Definition \ref{corr.terms.}).
The first two terms correspond to $\|u\|_{H^{s}}$, and $M_{s}^{(1)}(u)$, $M_{s}^{(2)}(u)$ and $M_{s}^{(3)}(u)$ are correction terms.
As defined in Definition \ref{corr.terms.},
we note that $\la_{j}(s)$ for $j=1,2,3,4$ is a linear polynomial in $s$.
The coefficient of $M_{s}^{(j)}(u)$ can be determined so that the time derivative of $M_{s}^{(j)}(u)$ cancels out $L_{j}(u)$ for $j=1,2$.
On the other hand, the time derivative of $M_{s}^{(1)}(u)$ also yields $L_{3}(u)$, that is,
\[\frac{d}{dt}M_{s}^{(1)}(u)\sim L_{1}(u)+L_{3}(u)\]
since $L_{1}(u)$ is the second order derivative loss.
Therefore, we need to collect coefficients of $L_{3}(u)$ resulting from both $\|D^{s}u\|$ and $M_{s}^{(1)}(u)$ when we determine the coefficient of $M_{s}^{(3)}(u)$.
For this reason, the coefficient of $M_{s}^{(3)}(u)$ is a quadratic polynomial in $s$.

Subsequently, using the conservation law corresponding to the $H^{4}$-norm of the solution,
we can obtain an {\it a priori} estimate of solutions in $H^{4}$.
Therefore, we can easily extend the solution obtained in Theorem \ref{main} globally.
Namely, we obtain the following result:

\begin{cor}
We write $M=\R$ or $\T$.
The Cauchy problem \eqref{BO1}--\eqref{BO2} with $c_{1}=3$, $-c_{2}=c_{5}=c_{6}=c_{7}=-2$ and $c_{3}=c_{4}=-1$
is globally well-posed in $H^{s}(M)$ for $s\ge4$.
\end{cor}

In what follows, we consider our problem only on $M=\T$, and the proof on $\R$ is alomst same as that on $\T$.
There are two differences, and one is the following:
\EQQS{
\HT(\HT f)(x)=
\begin{cases}
-f(x),&x\in\R,\\
-f(x)+\hat{f}(0),&x\in\T.
\end{cases}
}
However, such a difference does not yield difficulties in our argument since we have $|\hat{f}(0)|\le\|\hat{f}\|_{l^{\I}(\Z)}\le\|f\|_{L^{1}(\T)}\lesssim\|f\|_{L^{2}(\T)}$.
The other one is the Gagliardo-Nirenberg inequality (Lemma \ref{G.N.}), that is, we do not need to add $\|f\|_{L^{2}(\R)}$ on $\R$ when $l=0$.

This paper is organized as follows.
In Section 2, we prove the main result, admitting two Propositions \ref{ene.est.pr} and \ref{ene.est.0}.
In Section 3, we show the main estimate which is Proposition \ref{ene.est.pr}, that is, the energy inequality between two solutions in $H^{s}$.
In Section 4, we give a proof of the energy estimate in $L^{2}$ which is Proposition \ref{ene.est.0}.

\section{Proof of Theorem \ref{main}}
In this section we prove Theorem \ref{main}, admitting two propositions.
We denote the norm in $L^{p}(\T)$ by $\|\cdot\|_{p}$.
In particular, we simply write $\|\cdot\|:=\|\cdot\|_{2}$.
We denote $\|f\|_{H^{s}}:=2^{-1/2}(\|f\|^{2}+\|D^{s}f\|^{2})^{1/2}$ for a function $f$ and $s\ge0$,
where $D=\F^{-1}|\xi|\F$.
Let $\LR{\cdot,\cdot}:=\LR{\cdot,\cdot}_{L^{2}}$.
We also use the same symbol for $\LR{\cdot}:=(1+|\cdot|^{2})^{1/2}$.
Let $[A,B]=AB-BA$.
\begin{defn}
For a function $u$, we define
\EQQS{
&F_{1}(u):=\HT\ds^{3}u,\quad F_{2}(u):=c_{1}u\ds^{2}u+c_{2}(\ds u)^{2}+c_{3}(\HT\ds u)^{2}+c_{4}\HT(u\HT\ds^{2}u),\\
&F_{3}(u):=c_{5}\HT(u^{2}\ds u)+c_{6}u\HT(u\ds u)+c_{7}u^{2}\HT\ds u,\quad F_{4}(u):=-u^{4}.
}
Recall that $K(u)=F_{1}(u)+F_{2}(u)+F_{3}(u)+F_{4}(u)$.
\end{defn}

\begin{lem}\label{G.N.}
Assume that $l\in\mathbb{N}\cup \{0\}$ and $s\ge1$ satisfy $l\le s-1$ and a real number $p$ satisfies $2\le p\le\I$.
Put $\al=(l+1/2-1/p)/s$.
Then, we have
\begin{equation*}
\|\ds^{l}f\|_{p}\lesssim
\begin{cases}
\|f\|^{1-\al}\|D^{s}f\|^{\al}\quad (when\quad 1\le l\le s-1),\\
\|f\|^{1-\al}\|D^{s}f\|^{\al}+\|f\|\quad (when\quad l=0),
\end{cases}
\end{equation*}
for any $f\in H^{s}(\T)$.
\end{lem}

\begin{proof}
See Section 2 in \cite{SchwarzJr.} and Lemma 2.1 in \cite{Tanaka}.
\end{proof}

We employ the parabolic regularization:
\EQS{
\label{BO1pr}
&\dt u=\ds K(u)-\e\ds^{4}u,\\
\label{BO2pr}
&u(0,x)=\vp(x),
}
where $t\ge0$ and $\e>0$.
In what follows, we consider only $t\ge0$.
In the case $t\le0$, we only need to replace $-\e\ds^{4}u$ with $\e\ds^{4}u$ in \eqref{BO1pr}.
By the standard argument, we can establish the local well-posedness for \eqref{BO1pr}--\eqref{BO2pr} as follows.

\begin{prop}\label{para.regu.}
Let $s\ge3$ and $\e\in(0,1)$.
For any $\vp\in H^{s}(\T)$,
there exist $T_{\e}\in(0,\I]$ and the unique solution $u\in C([0,T_{\e}),H^{s}(\T))$
to the IVP \eqref{BO1pr}--\eqref{BO2pr} on $[0,T_{\e})$ such that
(i) $\liminf_{t\to T_{\e}}\|u(t)\|_{H^{3}}=\I$ or (ii) $T_{\e}=\I$ holds.
Moreover, we assume $\vp^{(j)}, \vp^{(\infty)}\in H^{s}(\T)$ satisfies
$\|\vp^{(j)}-\vp^{(\infty)}\|_{H^s}\to 0$ as $j\to \infty$.
Let $u^{(j)}$ (resp. $u^{(\infty)}$) $\in C([0,T_\e);H^{s}(\T))$
be the solution to \eqref{BO1pr}--\eqref{BO2pr} with initial data $\vp=\vp^{(j)}$ (resp. $\vp=\vp^{(\infty)}$).
Then, for any $T\in(0,T_\e)$, we have
$\sup_{t\in [0,T]} \|u^{(j)}(t)-u^{(\infty)}(t)\|_{H^s}\to 0$ as $j\to \infty$.
\end{prop}

\begin{proof}
See Proposition 2.8 in \cite{Tsugawa17} or Proposition 2.13 in \cite{Tanaka}.
\end{proof}

We construct a solution to \eqref{BO1}--\eqref{BO2} by a limiting procedure for solutions obtained by Proposition \ref{para.regu.}. In this argument, it is important to establish the time $T$ independent of $\e$, which is proved in Proposition \ref{ene.est.}.
For that purpose, we define the energy with correction terms in $H^{s}(\T)$.
As stated in Section 1, we note that the coefficient of $M_{s}^{(3)}$ is a quadratic polynomial in $s$.

\begin{defn}\label{corr.terms.}
Let $s\ge1$.
We define
\EQQS{
&\la_{1}(s):=(c_{1}-c_{4})s-\frac{c_{1}}{2}+2c_{2}+\frac{c_{4}}{2}, &&\la_{2}(s):=-2c_{3}s-c_{4},\\
&\la_{3}(s):=-2(c_{5}+c_{6}+c_{7})s-2c_{5}-c_{6}, &&\la_{4}(s):=2(c_{1}-c_{4})s-5c_{1}+4c_{2}+5c_{4}.
}
For functions $f,g\in H^{s}(\T)$, we also define
\EQQS{
E_{s}(f,g):=\frac{1}{2}\|f-g\|^{2}(1+C_{s}\|f\|^{2}+C_{s}\|f\|^{4s})
            +\frac{1}{2}\|D^{s}(f-g)\|^{2}+\sum_{j=1}^{3}M_{s}^{(j)}(f,g),
}
where
\EQQS{
&M_{s}^{(1)}(f,g):=\frac{\la_{1}(s)}{4}\int_{\T}f(\HT D^{s}(f-g))\HT D^{s-1}(f-g)dx,\\
&M_{s}^{(2)}(f,g):=\frac{\la_{2}(s)}{4}\int_{\T}(\HT\ds f)(D^{s-1}(f-g))^{2}dx,\\
&M_{s}^{(3)}(f,g):=\frac{\la_{1}(s)\la_{4}(s)+4\la_{3}(s)}{32}\int_{\T}f^{2}(D^{s-1}(f-g))^{2}dx
}
and $C_{s}$ is sufficiently large constant such that Lemma \ref{comparison} holds.
For simplicity, we write $E_{s}(f):=E_{s}(f,0)$ and $M_{s}^{(j)}(u):=M_{s}^{(j)}(u,0)$ for $j=1,2,3$.
\end{defn}

We define the energy with correction terms in $L^{2}(\T)$ since there is a problem to define $D^{-1}$ at very low frequency in $E_{0}(f,g)$.
For that purpose, we introduce the following.

\begin{defn}\label{J}
Let $\psi\in C^{\I}(\R)$ be a function satisfying $0\le \psi\le 1$ on $\R$ and
\EQQS{
\psi(\xi)=
\begin{cases}
1,\quad|\xi|\ge2,\\
0,\quad|\xi|\le1.
\end{cases}
}
We also define the operator
\[Jf(x):=\F^{-1}\left(\frac{\psi(\xi)}{|\xi|}\hat{f}(\xi)\right)(x)\]
for a function $f$.
\end{defn}

\begin{lem}\label{H^-1}
It holds that
\[\|Jf\|\le 2\|f\|_{H^{-1}}\]
for any $f\in H^{-1}(\T)$.
\end{lem}

\begin{proof}
This follows from the fact that $\LR{\xi}\le2|\xi|$ for $|\xi|\ge1$.
\end{proof}

\begin{defn}
For functions $f,g\in H^{1}(\T)$, we define
\EQQS{
E(f,g)
:=\frac{1}{2}\|f-g\|^{2}+\frac{1}{2}\|f-g\|_{H^{-1}}^{2}(1+C\|f\|^{2}+C\|f\|^{4})+\sum_{j=1}^{3}M^{(j)}(f,g),
}
where
\EQQS{
&M^{(1)}(f,g):=\frac{\la_{1}(0)}{4}\int_{\T}f(\HT(f-g))\HT J(f-g)dx,\\
&M^{(2)}(f,g):=\frac{\la_{2}(0)}{4}\int_{\T}(\HT\ds f)(J(f-g))^{2}dx,\\
&M^{(3)}(f,g):=\frac{\la_{1}(0)\la_{4}(0)+4\la_{3}(0)}{32}\int_{\T}f^{2}(J(f-g))^{2}dx
}
and $C$ is sufficiently large constant such that Lemma \ref{comparison2} holds.
\end{defn}



\begin{lem}\label{comparison}
Let $s\ge1$ and let $C_{s}>0$ be sufficiently large.
Then for any $f,g\in H^{s}(\T)$, it follows that
\EQQS{
E_{s}(f,g)\le\|f-g\|^{2}(1+C_{s}\|f\|^{2}+C_{s}\|f\|^{4s})+\|D^{s}(f-g)\|^{2}\le 4E_{s}(f,g).
}
\end{lem}

\begin{proof}
Lemma \ref{G.N.} shows that
\EQQS{
|M_{s}^{(2)}(f,g)|
&=\left|\frac{\la_{2}(s)}{4}\int_{\T}(\HT f)(D^{s-1}(f-g))D^{s-1}\ds(f-g)dx\right|\\
&\le C\|f\|\|D^{s}(f-g)\|\|\HT D^{s-1}(f-g)\|_{\I}\\
&\le C\|f\|\|f-g\|^{1/2s}\|D^{s}(f-g)\|^{2-1/2s}+C\|f\|\|f-g\|\|D^{s}(f-g)\|\\
&\le C\|f-g\|^{2}(\|f\|^{2}+\|f\|^{4s})+\frac{1}{12}\|D^{s}(f-g)\|^{2}.
}
Similarly, we can estimate $M_{s}^{(1)}(f,g)$ and $M_{s}^{(3)}(f,g)$ as follows:
\[|M_{s}^{(1)}(f,g)|,|M_{s}^{(3)}(f,g)|\le C\|f-g\|^{2}(\|f\|^{2}+\|f\|^{4s})+\frac{1}{12}\|D^{s}(f-g)\|^{2},\]
which completes the proof.
\end{proof}

A similar argument of the previous lemma together with Lemma \ref{H^-1} yields the following.

\begin{lem}\label{comparison2}
Let $C>0$ be sufficiently large.
Then for any $f,g\in H^{1}(\T)$, it follows that
\EQQS{
E(f,g)\le\|f-g\|_{H^{-1}}^{2}(1+C\|f\|^{2}+C\|f\|^{4})+\|f-g\|^{2}\le 4E(f,g).
}
\end{lem}

\begin{defn}
Let $s\ge0$.
For $f,g$, we define
\[I_{s}(f,g):=1+\|f\|_{H^{s}}+\|g\|_{H^{s}}.\]
\end{defn}

The main estimate in this paper is the following.

\begin{prop}\label{ene.est.pr}
Let $s\ge s_{0}>7/2$, $1\le s'\le s$, $\e_{j}\in(0,1)$, $\vp_{j}\in H^{s+4}(\T)$ and $u_{j}\in C([0,T_{\e_{j}});H^{s+4}(\T))$ be the solution to \eqref{BO1pr}--\eqref{BO2pr} obtained by Proposition \ref{para.regu.} with $\e=\e_{j}$ and $\vp=\vp_{j}$ for $j=1,2$.
Then there exists $C=C(s',s_{0})>0$ such that
\EQS{
\begin{aligned}\label{eq2.4}
&\frac{d}{dt}E_{s'}(u_{1}(t),u_{2}(t))\\
&\le CI_{s_{0}}(u_{1},u_{2})^{2(s'+2)}\{\|w\|_{H^{s'}}^{2}+\|w\|_{H^{s_{0}-3}}^{2}\|u_{2}\|_{H^{s'+3}}^{2}\\
&\quad+\|w\|_{H^{s_{0}}}^{2}(\|u_{1}\|_{H^{s'}}^{2}+\|u_{2}\|_{H^{s'}}^{2})\}
+\max\{\e_{1}^{2},\e_{2}^{2}\}\|u_{2}\|_{H^{s'+4}}^{2}
\end{aligned}
}
on $[0,\min\{T_{\e_{1}},T_{\e_{2}}\})$, where $w=u_{1}-u_{2}$.
\end{prop}

\begin{prop}\label{ene.est.0}
Let $s_{0}>7/2$, $T>0$ and $\e_{j}\in(0,1)$.
Let $u_{j}\in C([0,T];H^{s_{0}}(\T))\cap C((0,T);H^{s_{0}+1}(\T))$ satisfy \eqref{BO1pr} with $\e=\e_{j}$ on $[0,T]$ for $j=1,2$.
Then there esists $C=C(s_{0})>0$ such that
\EQS{
\begin{aligned}\label{eq2.5}
\frac{d}{dt}E(u_{1}(t),u_{2}(t))
\le CI_{s_{0}}(u_{1},u_{2})^{7}E(u_{1}(t),u_{2}(t))+\max\{\e_{1}^{2},\e_{2}^{2}\}\|u_{2}\|_{H^{s_{0}+1}}^{2}
\end{aligned}
}
on $[0,T]$, where $w=u_{1}-u_{2}$.
\end{prop}

If we admit Propositions \ref{ene.est.pr} and \ref{ene.est.0}, we can show the main result.
We prove Proposition \ref{ene.est.pr} (resp. Proposition \ref{ene.est.0}) in Section 3 (resp. Section 4).

\begin{prop}\label{ene.est.}
Let $s\ge s_{0}>7/2$,
$\e\in(0,1)$, $\vp\in H^{s}(\T)$.
Let 
$T_{\e}>0$ and let $u\in C([0,T_{\e}),H^{s}(\T))\cap C((0,T_{\e});H^{s+4}(\T))$
be the solution to \eqref{BO1pr}--\eqref{BO2pr},
both of which are obtained by Proposition \ref{para.regu.}.
Then, there exist $T=T(s_{0},\|\vp\|_{H^{s_{0}}})>0$ and $C=C(s,s_{0},\|\vp\|_{H^{s_{0}}})>0$ such that
\begin{align}\label{single1}
T_{\e}\ge T,\quad\sup_{t\in[0,T]}E_{s}(u(t))\le CE_{s}(\vp),\quad\frac{d}{dt}E_{s}(u(t))\le CE_{s}(u(t))
\end{align}
on $[0,T]$, where $T$ (resp. $C$) is monotone decreasing (resp. increasing) with $\|\vp\|_{H^{s_{0}}}$.
\end{prop}

\begin{proof}
Assume that the set $F=\{t\ge0; E_{s_{0}}(u(t))>2E_{s_{0}}(\vp)\}$ is not empty.
Set $T_{\e}^{*}=\inf F$.
Note that $0<T_{\e}^{*}\le T_{\e}$ and $E_{s_{0}}(u(t))\le2E_{s_{0}}(\vp)$ on $[0,T_{\e}^{*}]$.
Assume that there exists $t'\in[0,T_{\e}^{*}]$ such that
$E_{s_{0}}(u(t'))>2E_{s_{0}}(\vp)$.
This implies that $t'\ge T_{\e}^{*}$ by the definition of $T_{\e}^{*}$.
Then we have $t'=T_{\e}^{*}$.
Thus, $\sup_{t\in[0,T_{\e}^{*}]}E_{s_{0}}(u(t))\le C(\|\vp\|_{H^{s_{0}}})$ by $(ii)$ of Lemma \ref{comparison}.
By Proposition \ref{ene.est.pr} with $\vp_{2}=0$, there exists $C_{s}'=C(s,s_{0},\|\vp\|_{H^{s_{0}}})$ such that
\[\frac{d}{dt}E_{s}(u(t))\le C_{s}'E(u(t))\]
on $[0,T_{\e}^{*}]$.
The Gronwall inequality gives that
\begin{align}\label{eq3.62}
E_{s}(u(t))\le E_{s}(\vp)\exp(C_{s}'t)
\end{align}
on $[0,T_{\e}^{*}]$.
Here, we put $T=\min\{(2C_{s_{0}}')^{-1},T_{\e}^{*}\}$.
Then (\ref{eq3.62}) with $s=s_{0}$ shows that
\[E_{s_{0}}(u(t))\le E_{s_{0}}(\vp)\exp(2^{-1})<2E_{s_{0}}(\vp),\]
on $[0,T]$.
By the definition of $T_{\e}^{*}$ and the continuity of $E_{s_{0}}(u(t))$,
we obtain $0<T=(2C_{s_{0}}')^{-1}<T_{\e}^{*}\le T_{\e}$.
If $F$ is empty, then we have $T_{\e}^{*}=T_{\e}=\I$.
In particular, we can take $T=(2C_{s_{0}}')^{-1}<\I$, which concludes the proof.
\end{proof}

For the proof of the following uniqueness result, see Thorem 6.22 in \cite{Iorio}.

\begin{lem}[Uniqueness]\label{uniq.}
Let $s_{0}>7/2$, $\delta>0$ and
$u,v\in L^{\I}([0,\de];H^{s_{0}}(\T))$ satisfy \eqref{BO1} on $[0,\delta]$ with $u(0)=v(0)$
and satisfy
\[u,v\in C([0,\delta];H^{3}(\T))\cap C^{1}([0,\delta];H^{-1}(\T)).\]
Then $u\equiv v$ on $[0,\delta]$.
\end{lem}

It is important to employ the Bona-Smith type argument in the energy inequality for two solutions in $H^{s}$.
For that purpose, we introduce the following.

\begin{defn}\label{BS}
Let $s\ge0$, $f\in H^{s}(\T)$ and $\eta\in(0,1)$.
And let $\rho\in C_{0}^{\I}(\R)$ be $\rho(x):=1-\psi(x)$ for $x\in\R$.
We put
\[\ha{L_{\eta}f}(k):=\rho(\eta k)\hat{f}(k).\]
\end{defn}

For the proof of the following lemma, see Remark 3.5 in \cite{ErTzi}.

\begin{lem}\label{lem_BS}
Let $s\ge 0$, $\al\ge0$, $\eta\in(0,1)$ and $f\in H^{s}(\T)$.
Then, $L_{\eta} f \in H^\I(\T)$ satisfies
\EQQ{
&\|L_{\eta} f -f\|_{H^s} \to 0 \quad(\eta\to 0), \quad\|L_{\eta}f -f \|_{H^{s-\al}} \lec \ga^{\al} \|f\|_{H^s},\\
&\|L_{\eta}f\|_{H^{s-\al}}\le \|f\|_{H^{s-\al}}, \quad\|L_{\eta}f\|_{H^{s+\al}} \lec \ga^{-\al}\|f\|_{H^s}.
}
\end{lem}

\begin{proof}[Proof of Theorem \ref{main}]
We only need to prove Theorem \ref{main} for $t\ge0$ thanks to the transform $t\to -t$.
In what follows, without loss of generality, we may assume that $s_{0}$ is strictly smaller than $s$
since the assumption $\|\vp\|_{H^{s_{0}}}\le K$ is weaker than $\|\vp\|_{H^{s_{0}'}}\le K$
when $s_{0}<s_{0}'$.
First we prove the existence of the solution.
For $\vp\in H^{s}(\T)$, we put $\vp_{\eta}:=L_{\eta}\vp\in H^{\I}(\T)$ for $\eta\in(0,1)$.
By Proposition \ref{para.regu.},
there exists the unique solution $u_{\epsilon,\eta}\in C([0,T_{\e});H^{s}(\T))$ to \eqref{BO1pr}
with the initial data $\vp_{\eta}$ on $[0,T_{\e})$.
We see from Lemma \ref{lem_BS} that
\[\|\vp_{\eta}\|_{H^{s}}\le\|\vp\|_{H^{s}},\quad\|\vp_{\eta}\|_{H^{s_{0}}}\le\|\vp\|_{H^{s_{0}}}.\]
Then, Proposition \ref{ene.est.} with Lemma \ref{comparison} shows that there exists $T=T(s,s_{0},\|\vp\|_{H^{s_{0}}})>0$
such that
\[\sup_{t\in[0,T]}\|u_{\e,\eta}(t)\|_{H^{s}}\lesssim\sup_{t\in[0,T]} E_{s}(u_{\e,\eta}(t))^{1/2}
  \lesssim E_{s}(u_{\e,\eta}(0))^{1/2}\lesssim\|\vp_{\eta}\|_{H^{s}},\]
which implies that
\EQS{\label{eq2.6}
\sup_{t\in[0,T]}\|u_{\e,\eta}(t)\|_{H^{s+3}}\lesssim\eta^{-3}\|\vp\|_{H^{s}}.
}
Let $0<\e_{1}\le \e_{2}<1$ and $\eta_{j}=\e_{j}^{1/2s}$ for $j=1,2$.
Proposition \ref{ene.est.0} with $s'=s$ shows that there exists $C=C(s,s_{0},T,\|\vp\|_{H^{s_{0}}})>0$
such that
\EQQS{
\sup_{t\in[0,T]}\|u_{\e_{1},\eta_{1}}(t)-u_{\e_{2},\eta_{2}}(t)\|
&\le CE(u_{\e_{1},\eta_{1}}(0),u_{\e_{2},\eta_{2}}(0))\\
&\le C(\|\vp_{\eta_{1}}-\vp_{\eta_{2}}\|^{2}+\e_{2}^{2-2/s})^{1/2}
\le C\e_{2}^{1/2}.
}
By interpolation, it holds that for $\al\in[0,s]$,
\EQS{\label{eq2.7}
\sup_{t\in[0,T]}\|u_{\e_{1},\eta_{1}}(t)-u_{\e_{2},\eta_{2}}(t)\|_{H^{s-\al}}\lesssim \e_{2}^{\al/2s}.
}
Therefore, Proposition \ref{ene.est.pr} together with \eqref{eq2.6} and \eqref{eq2.7} shows that
\EQS{\label{eq2.8}
\sup_{t\in[0,T]}\|u_{\e_{1},\eta_{1}}(t)-u_{\e_{2},\eta_{2}}(t)\|_{H^{s}}
\lesssim\|\vp_{\eta_{1}}-\vp_{\eta_{2}}\|_{H^{s}}+\e_{2}^{(1-s_{0}/s)/2}
}
since $0<(1-s_{0}/s)/2<1-s_{0}/s<1-2/s$.
Then, $\{u_{\e,\eta}\}_{\e=\eta^{2s}}$ is a Cauchy sequence in $C([0,T];H^{s}(\T))$ as $\e\to0$
and there exists the limit $u\in C([0,T];H^{s}(\T))$.
It is clear that $u$ satisfy \eqref{BO1}--\eqref{BO2} on $[0,T]$.
We also note that letting $\e_{1}\to0$ in \eqref{eq2.8},
\EQS{\label{eq2.9}
\sup_{t\in[0,T]}\|u(t)-u_{\e}(t)\|_{H^{s}}
\lesssim\|\vp-\vp_{\e^{1/2s}}\|_{H^{s}}+\e^{(1-s_{0}/s)/2}
}
for $\e\in(0,1)$, where $u_{\e}:=u_{\e,\e^{1/2s}}$

Finally, we show the continuous dependence.
We claim that
if $\|\vp^{(j)}-\vp\|_{H^{s}}\to0$ as $j\to\I$, then $\sup_{t\in[0,T]}\|u^{(j)}(t)-u(t)\|_{H^{s}}\to0$ as $j\to\I$,
where $u^{(j)}$ (resp. $u$) is the solution to \eqref{BO1} with the initial data $\vp^{(j)}$ (resp. $\vp$) for $j\in\N$.
First note that the triangle inequality with Lemma \ref{lem_BS} gives that
\EQQS{
\|\vp^{(j)}-\vp_{\e^{1/2s}}^{(j)}\|_{H^{s}}
&\le\|\vp^{(j)}-\vp\|_{H^{s}}+\|\vp-\vp_{\e^{1/2s}}\|_{H^{s}}
   +\|\vp_{\e^{1/2s}}-\vp_{\e^{1/2s}}^{(j)}\|_{H^{s}}\\
&\lesssim\|\vp^{(j)}-\vp\|_{H^{s}}+\|\vp-\vp_{\e^{1/2s}}\|_{H^{s}}.
}
This together with \eqref{eq2.9} implies that
\EQQS{
&\sup_{t\in[0,T]}\|u^{(j)}(t)-u(t)\|_{H^{s}}\\
&\le\sup_{t\in[0,T]}\|u^{(j)}(t)-u_{\e}^{(j)}(t)\|_{H^{s}}+\sup_{t\in[0,T]}\|u_{\e}^{(j)}(t)-u_{\e}(t)\|_{H^{s}}
   +\sup_{t\in[0,T]}\|u_{\e}(t)-u(t)\|_{H^{s}}\\
&\le C\left(\|\vp^{(j)}-\vp\|_{H^{s}}+\|\vp-\vp_{\e^{1/2s}}\|_{H^{s}}
   +\sup_{t\in[0,T]}\|u_{\e}^{(j)}(t)-u_{\e}(t)\|_{H^{s}}+\e^{(1-s_{0}/s)/2}\right).
}
Let $\de>0$.
Then, there exists $\e_{0}\in(0,1)$ such that for any $\e\in(0,\e_{0})$
\EQQS{
C(\|\vp-\vp_{\e^{1/2s}}\|_{H^{s}}+\e^{(1-s_{0}/s)/2})<\frac{\de}{2}.
}
For each $\e\in(0,\e_{0})$,
we see from Proposition \ref{para.regu.} that there exists $N_{0}\in\N$ such that if $j>N_{0}$, then
\EQQS{
C\|\vp^{(j)}-\vp\|_{H^{s}}+C\sup_{t\in[0,T]}\|u_{\e}^{(j)}(t)-u_{\e}(t)\|_{H^{s}}<\frac{\de}{2},
}
which completes the proof of Theorem \ref{main}.
\end{proof}

\section{The energy estimate in $H^{s}$}
In this Section, we prove Proposition \ref{ene.est.pr}, which is the main estimate in this paper.
Before proving Proposition \ref{ene.est.pr}, we introduce some commutator estimates which are useful in evaluating nonlinear terms.

\begin{defn}
For $s\ge0$ and functions $f,g,h$ we define
\EQQS{
P_{s}^{(1)}(f,g)
&:=D^{s}\ds(f\ds^{2}g)-D^{s}\ds f\ds^{2}g-fD^{s}\ds^{3}g-(s+1)\ds fD^{s}\ds^{2}g\\
&\quad-\frac{s(s+1)}{2}\ds^{2}fD^{s}\ds g,\\
P_{s}^{(2)}(f,g)
&:=\HT D^{s}\ds(f\ds^{2}g)-(\HT D^{s}\ds f)\ds^{2}g-f\HT D^{s}\ds^{3}g\\
&\quad-(s+1)\ds f\HT D^{s}\ds^{2}g-\frac{s(s+1)}{2}\ds^{2}f\HT D^{s}\ds g,\\
P_{s}^{(3)}(f,g)
&:=D^{s}\ds(\ds f\ds g)-D^{s}\ds^{2}f\ds g-(s+1)D^{s}\ds f\ds^{2}g\\
&\quad-\ds fD^{s}\ds^{2}g-(s+1)\ds^{2}fD^{s}\ds g,\\
P_{s}^{(4)}(f,g)
&:=\HT D^{s}\ds(\ds f\ds g)-(\HT D^{s}\ds^{2}f)\ds g-(s+1)(\HT D^{s}\ds f)\ds^{2}g\\
&\quad-\ds f(\HT D^{s}\ds^{2}g)-(s+1)\ds^{2}f(\HT D^{s}\ds g),\\
P_{s}^{(5)}(f,g,h)
&:=D^{s}\ds(fg\ds h)-D^{s}\ds fg\ds h-fD^{s}\ds g\ds h-fgD^{s}\ds^{2}h\\
&\quad-(s+1)\ds fgD^{s}\ds h-(s+1)f\ds gD^{s}\ds h,\\
P_{s}^{(6)}(f,g,h)
&:=\HT D^{s}\ds(fg\ds h)-(\HT D^{s}\ds f)g\ds h-f(\HT D^{s}\ds g)\ds h-fg\HT D^{s}\ds^{2}h\\
&\quad-(s+1)\ds fg\HT D^{s}\ds h-(s+1)f\ds g\HT D^{s}\ds h,\\
P_{s}^{(7)}(f,g,h)
&:=D^{s}\ds(f\HT(g\ds h))-D^{s}\ds f\HT(g\ds h)-f(\HT D^{s}\ds g)\ds h-fg\HT D^{s}\ds^{2}h\\
&\quad-(s+1)\ds fg\HT D^{s}\ds h-(s+1)f\ds g\HT D^{s}\ds h,\\
P_{s}^{(8)}(f,g)
&:=D^{s}\ds(f\ds g)-D^{s}\ds f\ds g-fD^{s}\ds^{2}g-(s+1)\ds fD^{s}\ds g,\\
P_{s}^{(9)}(f,g)
&:=\HT D^{s}\ds(f\ds g)-(\HT D^{s}\ds f)\ds g-f\HT D^{s}\ds^{2}g-(s+1)\ds f\HT D^{s}\ds g.
}
\end{defn}

\begin{lem}\label{Leibniz}
Let $s_{0}>1/2$ and $s\ge0$.
Then
\[\|fg\|_{H^{s}}\lesssim\|f\|_{H^{s}}\|g\|_{H^{s_{0}}}+\|f\|_{H^{s_{0}}}\|g\|_{H^{s}}\]
for any $f,g\in H^{\max\{s_{0},s\}}(\T)$.
\end{lem}

\begin{proof}
This follows from the fact that $\LR{\xi}^{s}\lesssim\LR{\xi-\eta}^{s}+\LR{\eta}^{s}$ for any $\xi,\eta\in\Z$.
\end{proof}

For the proofs of the following three lemmas, see \cite{Tanaka}.

\begin{lem}\label{Leibniz2}
Let $s_{0}>3/2$ and $s\ge1$.
Then
\[\|[D^{s},f]\ds g\|\lesssim\|f\|_{H^{s}}\|g\|_{H^{s_{0}}}+\|f\|_{H^{s_{0}}}\|g\|_{H^{s}}\]
for any $f,g\in H^{\max\{s_{0},s\}}(\T)$.
\end{lem}

\begin{lem}\label{comm.est.H}
Let $s_{0}>1/2$, $k\in\N$ and $s_{1},s_{2}\ge0$.
Suppose that $s_{1}+s_{2}=k$.
Then, there exists $C=C(s_{0})>0$ such that for any $f\in H^{s_{0}+s_{1}}(\T)$ and $g\in H^{s_{2}}(\T)$
\[\|[\HT,f]\ds^{k}g\|\le C\|f\|_{H^{s_{0}+s_{1}}}\|g\|_{H^{s_{2}}}.\]
\end{lem}

\begin{lem}\label{comm.est.0}
Let $s_{0}>5/2$ and $s\ge0$.
Then there exists $C=C(s,s_{0})>0$ such that
\EQQS{
\|P_{s}^{(8)}(u,v)\|,\|P_{s}^{(9)}(u,v)\|
\le C(\|u\|_{H^{s_{0}}}\|v\|_{H^{s}}+\|u\|_{H^{s}}\|v\|_{H^{s_{0}}})
}
for any $u,v\in H^{\max\{s,s_{0}\}}(\T)$.
\end{lem}

\begin{lem}\label{comm.est.1}
Let $s_{0}>7/2$ and $s\ge0$.
Then there exists $C=C(s,s_{0})>0$ such that
\EQQS{
\|P_{s}^{(1)}(u,v)\|,\|P_{s}^{(2)}(u,v)\|\le C(\|u\|_{H^{s}}\|v\|_{H^{s_{0}}}+\|u\|_{H^{s_{0}}}\|v\|_{H^{s}})
}
for any $u,v\in H^{\max\{s,s_{0}\}}(\T)$.
\end{lem}

\begin{proof}
We show only the inequality for $P_{s}^{(1)}$.
The other one follows from a similar argument.
It suffices to show that
\EQS{\label{eq4.1}
\begin{aligned}
&\left||\xi|^{s}\xi\eta^{2}-|\xi-\eta|^{s}(\xi-\eta)\eta^{2}-|\eta|^{s}\eta^{3}-(s+1)(\xi-\eta)|\eta|^{s}\eta^{2}-\frac{s(s+1)}{2}(\xi-\eta)^{2}|\eta|^{s}\eta\right|\\
&\lesssim|\eta|^{3}|\xi-\eta|^{s}+|\eta|^{s}|\xi-\eta|^{3},
\end{aligned}
}
for any $\xi,\eta\in\Z$.
We split the summation region into three regions:
$R_{1}=\{3|\eta|\le|\xi-\eta|\},R_{2}=\{|\xi-\eta|/4\le|\eta|\le4|\xi-\eta|\}$ and $R_{3}=\{|\eta|\ge3|\xi-\eta|\}$.
On $R_{1}$, the mean value theorem shows that (\ref{eq4.1}) holds.
On $R_{2}$, it is obvious.
On $R_{3}$, it immediately follows that $|\xi-\eta|^{s+1}|\eta|^{2}\lesssim|\xi-\eta|^{s}|\eta|^{3}$.
Set $\sigma(x)=x|x|^{s}$ for $x\in\mathbb{R}$.
Note that $\sigma\in C^{3}(\R)$ when $s>2$.
The Taylor theorem shows that there exist $\tilde{\eta}\in(\xi,\eta)$ or $\tilde{\eta}\in(\eta,\xi)$ such that
\[\sigma(\xi)=\sigma(\eta)+\sigma'(\eta)(\xi-\eta)+\frac{\sigma''(\eta)}{2}(\xi-\eta)^{2}+\frac{\sigma'''(\ti{\eta})}{6}(\xi-\eta)^{3}.\]
This together with the fact that $|\tilde{\eta}|\sim|\xi|\sim|\eta|$ implies that (\ref{eq4.1}) holds.
When $1<s\le2$, \eqref{eq4.1} holds since
$|\xi-\eta|^{2}|\eta|^{s+1}=|\xi-\eta|^{2-s}|\xi-\eta|^{s}|\eta|^{s+1}\lesssim|\xi-\eta|^{s}|\eta|^{3}$ on $R_{3}$.
Similarly, when $0\le s\le1$, \eqref{eq4.1} holds by the above inequality with $|\xi-\eta||\eta|^{s+2}\lesssim|\xi-\eta|^{s}|\eta|^{3}$ on $R_{3}$, which concludes the proof.
\end{proof}

\begin{lem}\label{comm.est.2}
Let $s_{0}>7/2$ and $s\ge0$.
Then there exists $C=C(s,s_{0})>0$ such that
\EQQS{
\|P_{s}^{(3)}(u,v)\|,\|P_{s}^{(4)}(u,v)\|\le C(\|u\|_{H^{s}}\|v\|_{H^{s_{0}}}+\|u\|_{H^{s_{0}}}\|v\|_{H^{s}}),
}
for any $u,v\in H^{\max\{s,s_{0}\}}(\T)$.
\end{lem}

\begin{proof}
This follows from a similar argument of the previous lemma.
\end{proof}

\begin{lem}\label{comm.est.3}
Let $s_{0}>3/2$ and $s\ge0$.
Then there exists $C=C(s,s_{0})>0$ such that
\[\|\La_{s+1}(uv)-\La_{s+1}uv-u\La_{s+1}v\|\le C(\|u\|_{H^{s}}\|v\|_{H^{s_{0}}}+\|u\|_{H^{s_{0}}}\|v\|_{H^{s}})\]
for any $u,v\in H^{\max\{s,s_{0}\}}(\T)$, where $\La_{s+1}=D^{s}\ds$ or $D^{s+1}$.
\end{lem}

\begin{proof}
It suffices to show that for any $\xi,\eta\in\Z$
\[||\xi|^{s+1}-|\xi-\eta|^{s+1}-|\eta|^{s+1}|,||\xi|^{s}\xi-|\xi-\eta|^{s}(\xi-\eta)-|\eta|^{s}\eta|\lesssim|\eta||\xi-\eta|^{s}+|\xi-\eta||\eta|^{s}.\]
If $s=0$, then it is obvious by the triangle inequality.
In the case $s>0$, this follows from a similar argument of the proof of Lemma \ref{comm.est.1}.
\end{proof}

\begin{lem}\label{comm.est.4}
Let $s_{0}>5/2$ and $s\ge0$.
Then there exists $C=C(s,s_{0})>0$ such that
\EQQS{
&\|P_{s}^{(5)}(u_{1},u_{2},u_{3})\|,\|P_{s}^{(6)}(u_{1},u_{2},u_{3})\|\\
&\le C(\|u_{1}\|_{H^{s}}\|u_{2}\|_{H^{s_{0}}}\|u_{3}\|_{H^{s_{0}}}+\|u_{1}\|_{H^{s_{0}}}\|u_{2}\|_{H^{s}}\|u_{3}\|_{H^{s_{0}}}+\|u_{1}\|_{H^{s_{0}}}\|u_{2}\|_{H^{s_{0}}}\|u_{3}\|_{H^{s}})
}
for any $u_{j}\in H^{\max\{s,s_{0}\}}(\T)$ for $j=1,2,3$.
\end{lem}

\begin{proof}
We show only the inequality for $P_{s}^{(5)}$.
The other one follows from a similar argument.
Aplying Lemma \ref{Leibniz}, \ref{comm.est.0} and \ref{comm.est.3}, we have
\EQQS{
&\|P_{s}^{(5)}(u_{1},u_{2},u_{3})\|\\
&\le\|P_{s}^{(8)}(u_{1}u_{2},u_{3})\|+\|D^{s}\ds(u_{1}u_{2})-D^{s}\ds u_{1}u_{2}-u_{1}D^{s}\ds u_{2}\|\|\ds u_{3}\|_{\I},\\
&\lesssim\|u_{1}\|_{H^{s}}\|u_{2}\|_{H^{s_{0}}}\|u_{3}\|_{H^{s_{0}}}+\|u_{1}\|_{H^{s_{0}}}\|u_{2}\|_{H^{s}}\|u_{3}\|_{H^{s_{0}}}+\|u_{1}\|_{H^{s_{0}}}\|u_{2}\|_{H^{s_{0}}}\|u_{3}\|_{H^{s}},
}
which completes the proof.
\end{proof}

\begin{lem}\label{comm.est.5}
Let $s_{0}>5/2$ and $s\ge0$.
Then there exists $C=C(s,s_{0})>0$ such that
\EQQS{
&\|P_{s}^{(7)}(u_{1},u_{2},u_{3})\|\\
&\le C(\|u_{1}\|_{H^{s}}\|u_{2}\|_{H^{s_{0}}}\|u_{3}\|_{H^{s_{0}}}+\|u_{1}\|_{H^{s_{0}}}\|u_{2}\|_{H^{s}}\|u_{3}\|_{H^{s_{0}}}+\|u_{1}\|_{H^{s_{0}}}\|u_{2}\|_{H^{s_{0}}}\|u_{3}\|_{H^{s}})
}
for any $u_{j}\in H^{\max\{s,s_{0}\}}(\T)$ for $j=1,2,3$.
\end{lem}

\begin{proof}
We see from the proof of Lemma \ref{comm.est.0} and \ref{comm.est.3} that for any $\xi,\xi_{1},\xi_{2}\in\Z$
\EQQS{
&||\xi|^{s}\xi\xi_{2}-|\xi-\xi_{1}-\xi_{2}|^{s}(\xi-\xi_{1}-\xi_{2})\xi_{2}-|\xi_{1}|^{s}\xi_{1}\xi_{2}\\
&\quad-|\xi_{2}|^{s}\xi_{2}^{2}-(s+1)(\xi-\xi_{1}-\xi_{2})|\xi_{2}|^{s}\xi_{2}-(s+1)\xi_{1}|\xi_{2}|^{s}\xi_{2}|\\
&\le||\xi|^{s}\xi\xi_{2}-|\xi-\xi_{2}|^{s}(\xi-\xi_{2})\xi_{2}-|\xi_{2}|^{s}\xi_{2}^{2}
  -(s+1)(\xi-\xi_{2})|\xi_{2}|^{s}\xi_{2}|\\
&\quad+||\xi-\xi_{2}|^{s}(\xi-\xi_{2})-|\xi_{1}|^{s}\xi_{1}-|\xi-\xi_{1}-\xi_{2}|^{s}(\xi-\xi_{1}-\xi_{2})||\xi_{2}|
\lesssim \Xi(\xi,\xi_{1},\xi_{2}),
}
where
\EQQS{
\Xi(\xi,\xi_{1},\xi_{2}):=|\xi-\xi_{2}|^{2}|\xi_{2}|^{s}+|\xi-\xi_{2}|^{s}|\xi_{2}|^{2}+(|\xi-\xi_{1}-\xi_{2}|^{s}|\xi_{1}|+|\xi-\xi_{1}-\xi_{2}||\xi_{1}|^{s})|\xi_{2}|.
}
We see from Lemma \ref{comm.est.H} that
\EQQS{
&\|P_{s}^{(7)}(u_{1},u_{2},u_{3})\|\\
&\le\|D^{s}\ds(u_{1}\HT(u_{2}\ds u_{3}))-D^{s}\ds u_{1}\HT(u_{2}\ds u_{3})-u_{1}\HT(D^{s}\ds u_{2}\ds u_{3})\\
&\quad-u_{1}\HT(u_{2}D^{s}\ds^{2}u_{3})-(s+1)\ds u_{1}\HT(u_{2}D^{s}\ds u_{3})
      -(s+1)u_{1}\HT(\ds u_{2}D^{s}\ds u_{3})\|\\
&\quad+\|u_{1}\|_{\I}(\|[\HT,\ds u_{3}]D^{s}\ds u_{2}\|+\|[\HT,u_{2}]D^{s}\ds^{2}u_{3}\|)\\
&\quad+(s+1)(\|\ds u_{1}\|_{\I}\|[\HT,u_{2}]D^{s}\ds u_{3}\|+\|u_{1}\|_{\I}\|[\HT,\ds u_{2}]D^{s}\ds u_{3}\|)\\
&\lesssim\left\|\sum_{\xi_{1},\xi_{2}}\Xi(\xi,\xi_{1},\xi_{2})
         |\hat{u_{1}}(\xi-\xi_{1}-\xi_{2})||\hat{u_{2}}(\xi_{1})||\hat{u_{3}}(\xi_{2})|\right\|_{l^{2}_{\xi}}\\
&\quad+\|u_{1}\|_{H^{s_{0}}}\|u_{2}\|_{H^{s_{0}}}\|u_{3}\|_{H^{s}}+\|u_{1}\|_{H^{s_{0}}}\|u_{2}\|_{H^{s}}\|u_{3}\|_{H^{s_{0}}}\\
&\lesssim\|u_{1}\|_{H^{s}}\|u_{2}\|_{H^{s_{0}}}\|u_{3}\|_{H^{s_{0}}}+\|u_{1}\|_{H^{s_{0}}}\|u_{2}\|_{H^{s}}\|u_{3}\|_{H^{s_{0}}}+\|u_{1}\|_{H^{s_{0}}}\|u_{2}\|_{H^{s_{0}}}\|u_{3}\|_{H^{s}},
}
which completes the proof.
\end{proof}

\begin{lem}\label{comm.est.7}
Let $s_{0}>5/2$ and $s\ge0$.
Then there exists $C=C(s,s_{0})>0$ such that
\[\|\La_{s}(u\ds^{2}v)-u\La_{s}\ds^{2}v-s\ds u\La_{s}\ds v\|
  \le C(\|u\|_{H^{s_{0}}}\|v\|_{H^{s}}+\|u\|_{H^{s}}\|v\|_{H^{s_{0}}})\]
for any $u,v\in H^{\max\{s,s_{0}\}}(\T)$, where $\La_{s}=D^{s}$ or $\HT D^{s}$.
\end{lem}

\begin{proof}
The proof is very similar to that of Lemma \ref{comm.est.1}.
\end{proof}



\begin{lem}\label{comm.est.9}
Let $s_{0}>3/2$ and $s\ge0$.
Then there exists $C=C(s,s_{0})>0$ such that
\EQQS{
&\|\La_{s}(u_{1}\HT(u_{2}\ds u_{3}))-u_{1}u_{2}\HT\La_{s}\ds u_{3}\|\\
&\le C(\|u_{1}\|_{H^{s}}\|u_{2}\|_{H^{s_{0}}}\|u_{3}\|_{H^{s_{0}}}+\|u_{1}\|_{H^{s_{0}}}\|u_{2}\|_{H^{s}}\|u_{3}\|_{H^{s_{0}}}+\|u_{1}\|_{H^{s_{0}}}\|u_{2}\|_{H^{s_{0}}}\|u_{3}\|_{H^{s}})
}
for any $u_{j}\in H^{\max\{s,s_{0}\}}(\T)$ for $j=1,2,3$,
where $\La_{s}=D^{s}$ or $\HT D^{s}$.
\end{lem}

\begin{proof}
It suffices to show that
\EQQS{
&\|\La_{s}(u_{1}\HT(u_{2}\ds u_{3}))-u_{1}\HT(u_{2}\La_{s}\ds u_{3})\|\\
&\lesssim\|u_{1}\|_{H^{s}}\|u_{2}\|_{H^{s_{0}}}\|u_{3}\|_{H^{s_{0}}}+\|u_{1}\|_{H^{s_{0}}}\|u_{2}\|_{H^{s}}\|u_{3}\|_{H^{s_{0}}}+\|u_{1}\|_{H^{s_{0}}}\|u_{2}\|_{H^{s_{0}}}\|u_{3}\|_{H^{s}}.
}
Indeed, Lemma \ref{comm.est.H} shows that
\EQQS{
\|u_{1}[\HT,u_{2}]\La_{s}\ds u_{3}\|
\lesssim\|u_{1}\|_{H^{s_{0}-1}}\|u_{2}\|_{H^{s_{0}}}\|u_{3}\|_{H^{s}}.
}
The standard argument implies that
\EQQS{
||\xi|^{s}\xi_{2}-|\xi_{2}|^{s}\xi_{2}|,|\sgn(\xi)|\xi|^{s}\xi_{2}-\sgn(\xi_{2})|\xi_{2}|^{s}\xi_{2}|\lesssim|\xi-\xi_{2}|^{s}|\xi_{2}|+|\xi-\xi_{2}||\xi_{s}|^{s},
}
which completes the proof by the triangle inequality.
\end{proof}

\begin{lem}\label{lem4.4}
Let $s\ge0$ and $s_{0}>5/2$.
Let $u\in H^{s_{0}}(\T)$ and $w\in H^{s+2}(\T)$.
Then
\EQQS{
|\LR{uD^{s}\ds^{2}w,D^{s}w}+\LR{u,(D^{s}\ds w)^{2}}|
\lesssim\|u\|_{H^{s_{0}}}\|w\|_{H^{s}}^{2}
}
\end{lem}

\begin{proof}
Note that
\[\LR{uD^{s}\ds^{2}w,D^{s}w}=\frac{1}{2}\LR{\ds^{2}u,(D^{s}w)^{2}}-\LR{u,(D^{s}\ds w)^{2}},\]
which shows the claim.
\end{proof}

As stated in Section 1, by the integration by parts, the third order derivative loss can be reduced to the second order one.

\begin{lem}\label{lem4.5}
Let $s\ge0$ and $s_{0}>7/2$.
Let $u\in H^{s_{0}}(\T)$ and $w\in H^{s+3}(\T)$.
Then
\EQQS{
\left|\LR{uD^{s}\ds^{3}w,D^{s}w}-\frac{3}{2}\LR{\ds u,(D^{s}\ds w)^{2}}\right|
\lesssim\|u\|_{H^{s_{0}}}\|w\|_{H^{s}}^{2}
}
\end{lem}

\begin{proof}
Note that
\EQQS{
\LR{uD^{s}\ds^{3}w,D^{s}w}
=-\LR{\ds uD^{s}\ds^{2}w,D^{s}w}+\frac{1}{2}\LR{\ds u,(D^{s}\ds w)^{2}},
}
which together with Lemma \ref{lem4.4} shows the claim.
\end{proof}

\begin{lem}\label{reduction}
Let $s_{0}>1/2$ and $u,v$ be sufficiently smooth function defined on $\T$.
Then there exists $C=C(s_{0})>0$ such that
\[|\LR{\ds(v\HT\ds u),u}|\le C\|v\|_{H^{s_{0}+2}}\|u\|^{2}.\]
\end{lem}

\begin{proof}
See Lemma 2.6 in \cite{Tanaka}.
\end{proof}

\begin{lem}\label{reduction2new}
Let $s\ge0$ and $s_{0}>5/2$.
Let $u\in H^{s_{0}}(\T)$ and $w\in H^{s+2}(\T)$.
Then
\EQQS{
|\LR{u,(\HT D^{s}\ds w)^{2}}-\LR{u,(D^{s}\ds w)^{2}}|
\lesssim\|u\|_{H^{s_{0}}}\|w\|_{H^{s}}^{2}.
}
\end{lem}

\begin{proof}
We have
\EQQS{
\LR{u,(\HT D^{s}\ds w)^{2}}
=\LR{\HT(\ds u\HT D^{s}\ds w),D^{s}w}+\LR{\HT(u\HT D^{s}\ds^{2}w),D^{s}w}=:A+B.
}
For $A$, note that
\EQQS{
|A|\le|\LR{[\HT,\ds u]\HT D^{s}\ds w,D^{s}w}|+|\LR{\ds uD^{s}\ds w,D^{s}w}|
\lesssim\|u\|_{H^{s_{0}}}\|w\|_{H^{s}}^{2}.
}
For $B$, we have
\EQQS{
B
&=\LR{[\HT,u]\HT D^{s}\ds^{2}w,D^{s}w}-\LR{uD^{s}\ds^{2}w,D^{s}w}\\
&=\LR{[\HT,u]\HT D^{s}\ds^{2}w,D^{s}w}-\frac{1}{2}\LR{\ds^{2}u,(D^{s}w)^{2}}+\LR{u,(D^{s}\ds w)^{2}},
}
which concludes the proof.
\end{proof}

We are ready to evaluate nonlinear terms.
First, we estimate terms in $F_{2}(u)$.

\begin{lem}\label{lem2.3new}
Let $s_{0}>7/2$ and $s\ge0$.
Let $u,v\in H^{\max\{s+3,s_{0}\}}(\T)$.
Then
\EQQS{
&|\LR{D^{s}\ds (u\ds^{2}u-v\ds^{2}v),D^{s}w}+(s-1/2)\LR{\ds u,(D^{s}\ds w)^{2}}|\\
&\quad\lesssim I_{s_{0}}(u,v)\{\|w\|_{H^{s}}^{2}+\|w\|_{H^{s_{0}-3}}^{2}\|v\|_{H^{s+3}}^{2}+(\|u\|_{H^{s}}^{2}+\|v\|_{H^{s}}^{2})\|w\|_{H^{s_{0}}}^{2}\},
}
where $w=u-v$.
\end{lem}

\begin{proof}
Set
\[\LR{D^{s}\ds (u\ds^{2}u-v\ds^{2}v),D^{s}w}=\LR{D^{s}\ds (u\ds^{2}w),D^{s}w}+\LR{D^{s}\ds (w\ds^{2}v),D^{s}w}
  =:R_{1}+R_{2}.\]
Lemma \ref{comm.est.1}, \ref{lem4.4} and \ref{lem4.5} show that
\EQQS{
&|R_{1}+(s-1/2)\LR{\ds u,(D^{s}\ds w)^{2}}|\\
&\le|\LR{P_{s}^{(1)}(u,w),D^{s}w}|+|\LR{D^{s}\ds(w+v)\ds^{2}w,D^{s}w}|+|\LR{uD^{s}\ds^{3}w,D^{s}w}-(3/2)\LR{\ds u,(D^{s}\ds w)^{2}}|\\
&\quad+(s+1)|\LR{\ds uD^{s}\ds^{2}w,D^{s}w}+\LR{\ds u,(D^{s}\ds w)^{2}}|+s(s+1)|\LR{\ds^{3}u,(D^{s}w)^{2}}|/4\\
&\lesssim\|w\|_{H^{s}}(\|w\|_{H^{s_{0}-1}}\|v\|_{H^{s+1}}+\|u\|_{H^{s_{0}}}\|w\|_{H^{s}}+\|w\|_{H^{s_{0}}}\|u\|_{H^{s}}
         +\|w\|_{H^{s_{0}}}\|w\|_{H^{s}}).
}
We see from a similar argument that
\EQQS{
|R_{2}|
&\lesssim I_{s_{0}}(u,v)\{\|w\|_{H^{s}}^{2}+\|w\|_{H^{s_{0}-3}}^{2}\|v\|_{H^{s+3}}^{2}+(\|u\|_{H^{s}}^{2}+\|v\|_{H^{s}}^{2})\|w\|_{H^{s_{0}}}^{2}\},
}
which completes the proof.
\end{proof}

\begin{lem}\label{lem2.4new}
Let $s_{0}>7/2$ and $s\ge0$.
Let $u,v\in H^{\max\{s+2,s_{0}\}}(\T)$.
Then
\EQQS{
&|\LR{D^{s}\ds((\ds u)^{2}-(\ds v)^{2}),D^{s}w}+2\LR{\ds u,(D^{s}\ds w)^{2}}|\\
&\quad\lesssim I_{s_{0}}(u,v)\{\|w\|_{H^{s}}^{2}+\|w\|_{H^{s_{0}-2}}^{2}\|v\|_{H^{s+2}}^{2}+(\|u\|_{H^{s}}^{2}+\|v\|_{H^{s}}^{2})\|w\|_{H^{s_{0}}}^{2}\},
}
where $w=u-v$.
\end{lem}

\begin{proof}
Set $z=u+v$.
Note that
\EQQS{
&\LR{D^{s}\ds^{2}z\ds w,D^{s}w}+\LR{\ds zD^{s}\ds^{2}w,D^{s}w}\\
&=\LR{\ds wD^{s}\ds^{2}w,D^{s}w}+2\LR{\ds wD^{s}\ds^{2}v,D^{s}w}+\LR{\ds zD^{s}\ds^{2}w,D^{s}w}\\
&=2\LR{\ds uD^{s}\ds^{2}w,D^{s}w}+2\LR{\ds wD^{s}\ds^{2}v,D^{s}w}.
}
Lemma \ref{comm.est.2} and \ref{lem4.4} show that
\EQQS{
&|\LR{D^{s}\ds(\ds z\ds w),D^{s}w}+2\LR{\ds u,(D^{s}\ds w)^{2}}|\\
&\le|\LR{P_{s}^{(3)}(z,w),D^{s}w}|+(s+1)|\LR{\ds^{3}u,(D^{s}w)^{2}}|+2(s+1)|\LR{\ds wD^{s}\ds v,D^{s}w}|\\
&\quad+2|\LR{\ds uD^{s}\ds^{2}w,D^{s}w}+\LR{\ds u,(D^{s}\ds w)^{2}}|+2|\LR{\ds wD^{s}\ds^{2}v,D^{s}w}|\\
&\lesssim I_{s_{0}}(u,v)\{\|w\|_{H^{s}}^{2}+\|w\|_{H^{s_{0}-2}}^{2}\|v\|_{H^{s+2}}^{2}+(\|u\|_{H^{s}}^{2}+\|v\|_{H^{s}}^{2})\|w\|_{H^{s_{0}}}^{2}\},
}
which completes the proof.
\end{proof}

\begin{lem}\label{lem2.5new}
Let $s_{0}>7/2$ and $s\ge0$.
Let $u,v\in H^{\max\{s+2,s_{0}\}}(\T)$.
Then
\EQQS{
&|\LR{D^{s}\ds((\HT\ds u)^{2}-(\HT\ds v)^{2}),D^{s}w}-2s\LR{(\HT\ds^{2}u)\HT D^{s}\ds w,D^{s}w}|\\
&\lesssim I_{s_{0}}(u,v)\{\|w\|_{H^{s}}^{2}+\|w\|_{H^{s_{0}-2}}^{2}\|v\|_{H^{s+2}}^{2}+(\|u\|_{H^{s}}^{2}+\|v\|_{H^{s}}^{2})\|w\|_{H^{s_{0}}}^{2}\},
}
where $w=u-v$.
\end{lem}

\begin{proof}
Set $z=u+v$.
As in the proof of Lemma \ref{lem2.4new}, we have
\EQQS{
&\LR{(\HT\ds w)\HT D^{s}\ds^{2}z,D^{s}w}+\LR{(\HT\ds z)\HT D^{s}\ds^{2}w,D^{s}w}\\
&=2\LR{(\HT\ds u)\HT D^{s}\ds^{2}w,D^{s}w}+2\LR{(\HT\ds w)\HT D^{s}\ds^{2}v,D^{s}w}
}
and
\EQQS{
&\LR{(\HT\ds^{2}w)\HT D^{s}\ds z,D^{s}w}+\LR{(\HT\ds^{2}z)\HT D^{s}\ds w,D^{s}w}\\
&=2\LR{(\HT\ds^{2}u)\HT D^{s}\ds w,D^{s}w}+2\LR{(\HT\ds^{2}w)\HT D^{s}\ds v,D^{s}w}.
}
Then Lemma \ref{comm.est.2}, \ref{lem4.4} and \ref{reduction} show that
\EQQS{
&|\LR{D^{s}\ds((\HT\ds z)\HT\ds w),D^{s}w}-2s\LR{(\HT\ds^{2}u)\HT D^{s}\ds w,D^{s}w}|\\
&\le|\LR{P_{s}^{(3)}(\HT z,\HT w),D^{s}w}|+2|\LR{\ds((\HT\ds u)\HT D^{s}\ds w),D^{s}w}|\\
&\quad+2|\LR{(\HT\ds w)\HT D^{s}\ds^{2}v,D^{s}w}|+2(s+1)|\LR{(\HT\ds w)\HT D^{s}\ds v,D^{s}w}|\\
&\lesssim I_{s_{0}}(u,v)\{\|w\|_{H^{s}}^{2}+\|w\|_{H^{s_{0}-2}}^{2}\|v\|_{H^{s+2}}^{2}+(\|u\|_{H^{s}}^{2}+\|v\|_{H^{s}}^{2})\|w\|_{H^{s_{0}}}^{2}\},
}
which completes the proof.
\end{proof}

\begin{lem}\label{lem2.6new}
Let $s_{0}>7/2$ and $s\ge0$.
Let $u,v\in H^{\max\{s+3,s_{0}\}}(\T)$.
Then
\EQQS{
&|\LR{\HT D^{s}\ds(u\HT\ds^{2}u-v\HT\ds^{2}v),D^{s}w}-\LR{(\HT\ds^{2}u)\HT D^{s}\ds w,D^{s}w}
 -(s-1/2)\LR{\ds u,(D^{s}\ds w)^{2}}|\\
&\lesssim I_{s_{0}}(u,v)\{\|w\|_{H^{s}}^{2}+\|w\|_{H^{s_{0}-3}}^{2}\|v\|_{H^{s+3}}^{2}+(\|u\|_{H^{s}}^{2}+\|v\|_{H^{s}}^{2})\|w\|_{H^{s_{0}}}^{2}\},
}
where $w=u-v$.
\end{lem}

\begin{proof}
Set $z=u+v$.
Note that
\EQQS{
&\LR{(\HT D^{s}\ds u)\HT\ds^{2}w,D^{s}w}+\LR{(\HT D^{s}\ds w)\HT\ds^{2}v,D^{s}w}\\
&=\LR{(\HT D^{s}\ds w)\HT\ds^{2}w,D^{s}w}+\LR{(\HT D^{s}\ds w)\HT\ds^{2}v,D^{s}w}
  +\LR{(\HT D^{s}\ds v)\HT\ds^{2}w,D^{s}w}\\
&=\LR{(\HT\ds^{2}u)\HT D^{s}\ds w,D^{s}w}+\LR{(\HT\ds^{2}w)\HT D^{s}\ds v,D^{s}w}.
}
Lemma \ref{comm.est.1}, \ref{lem4.4} and \ref{lem4.5} show that
\EQQS{
&|\LR{\HT D^{s}\ds(u\HT\ds^{2}u-v\HT\ds^{2}v),D^{s}w}-\LR{(\HT\ds^{2}u)\HT D^{s}\ds w,D^{s}w}-(s-1/2)\LR{\ds u,(D^{s}\ds w)^{2}}|\\
&\le|\LR{P_{s}^{(2)}(u,\HT w)+P_{s}^{(2)}(w,\HT v),D^{s}w}|+|\LR{uD^{s}\ds^{3}w,D^{s}w}-(3/2)\LR{\ds u,(D^{s}\ds w)^{2}}|\\
&\quad+(s+1)|\LR{\ds uD^{s}\ds^{2}w,D^{s}w}+\LR{\ds u,(D^{s}\ds w)^{2}}|+|\LR{wD^{s}\ds^{3}v,D^{s}w}|\\
&\quad+s(s+1)|\LR{\ds^{3}u,(D^{s}w)^{2}}|/4+(s+1)|\LR{\ds wD^{s}\ds^{2}v,D^{s}w}|+s(s+1)|\LR{\ds^{2}wD^{s}\ds v,D^{s}w}|/2\\
&\quad+|\LR{(\HT\ds^{2}w)\HT D^{s}\ds v,D^{s}w}|\\
&\lesssim I_{s_{0}}(u,v)\{\|w\|_{H^{s}}^{2}+\|w\|_{H^{s_{0}-3}}^{2}\|v\|_{H^{s+3}}^{2}+(\|u\|_{H^{s}}^{2}+\|v\|_{H^{s}}^{2})\|w\|_{H^{s_{0}}}^{2}\},
}
which completes the proof.
\end{proof}

Next, we estimate nonlinear terms in $F_{3}(u)$ and $F_{4}(u)$.

\begin{lem}\label{lem2.8new}
Let $s_{0}>7/2$ and $s\ge0$.
Let $u,v\in H^{\max\{s+2,s_{0}\}}(\T)$.
Then
\EQQS{
&|\LR{\HT D^{s}\ds(u^{2}\ds u-v^{2}\ds v),D^{s}w}-2(s+1)\LR{u\ds u\HT D^{s}\ds w,D^{s}w}|\\
&\lesssim I_{s_{0}}(u,v)^{2}\{\|w\|_{H^{s}}^{2}+\|w\|_{H^{s_{0}-2}}^{2}\|v\|_{H^{s+2}}^{2}+(\|u\|_{H^{s}}^{2}+\|v\|_{H^{s}}^{2})\|w\|_{H^{s_{0}}}^{2}\},
}
where $w=u-v$.
\end{lem}

\begin{proof}
Set $z=u+v$.
Note that
\EQQS{
&2\LR{u\ds w\HT D^{s}\ds u,D^{s}w}+\LR{w\ds v\HT D^{s}\ds z,D^{s}w}+\LR{z\ds v\HT D^{s}\ds w,D^{s}w}\\
&=2\LR{u\ds u\HT D^{s}\ds w,D^{s}w}+2\LR{u\ds w\HT D^{s}\ds v,D^{s}w}+2\LR{w\ds v\HT D^{s}\ds v,D^{s}w}
}
and that $u^{2}\ds u-v^{2}\ds v=u^{2}\ds w+zw\ds v$.
Lemma \ref{comm.est.4}, \ref{comm.est.H} and \ref{reduction} show that
\EQQS{
&|\LR{\HT D^{s}\ds(u^{2}\ds w+zw\ds v),D^{s}w}-2(s+1)\LR{u\ds u\HT D^{s}\ds w,D^{s}w}|\\
&\le|\LR{P_{s}^{(6)}(u,u,w)+P_{s}^{(6)}(z,w,v),D^{s}w}|+2|\LR{u\ds w\HT D^{s}\ds v,D^{s}w}|\\
&\quad+2|\LR{w\ds v\HT D^{s}\ds v,D^{s}w}|+|\LR{\ds(u^{2}\HT D^{s}\ds w),D^{s}w}|\\
&\quad+|\LR{zw\HT D^{s}\ds^{2}v,D^{s}w}|+(s+1)|\LR{\ds(zw)\HT D^{s}\ds v,D^{s}w}|\\
&\lesssim I_{s_{0}}(u,v)^{2}\{\|w\|_{H^{s}}^{2}+\|w\|_{H^{s_{0}-2}}^{2}\|v\|_{H^{s+2}}^{2}+(\|u\|_{H^{s}}^{2}+\|v\|_{H^{s}}^{2})\|w\|_{H^{s_{0}}}^{2}\},
}
which completes the proof.
\end{proof}

\begin{lem}\label{lem2.9new}
Let $s_{0}>7/2$ and $s\ge0$.
Let $u,v\in H^{\max\{s+2,s_{0}\}}(\T)$.
Then
\EQQS{
&|\LR{D^{s}\ds(u\HT(u\ds u)-v\HT(v\ds v)),D^{s}w}-(2s+1)\LR{u\ds u\HT D^{s}\ds w,D^{s}w}|\\
&\lesssim I_{s_{0}}(u,v)^{2}\{\|w\|_{H^{s}}^{2}+\|w\|_{H^{s_{0}-2}}^{2}\|v\|_{H^{s+2}}^{2}+(\|u\|_{H^{s}}^{2}+\|v\|_{H^{s}}^{2})\|w\|_{H^{s_{0}}}^{2}\},
}
where $w=u-v$.
\end{lem}

\begin{proof}
Note that
\EQQS{
&\LR{u\ds w\HT D^{s}\ds u,D^{s}w}+\LR{u\ds v\HT D^{s}\ds w,D^{s}w}\\
&=\LR{u\ds u\HT D^{s}\ds w,D^{s}w}+\LR{u\ds w\HT D^{s}\ds v,D^{s}w}.
}
Then Lemma \ref{comm.est.5} and \ref{reduction} show that
\EQQS{
&|\LR{D^{s}\ds(u\HT(u\ds w)+u\HT(w\ds v)+w\HT(v\ds v)),D^{s}w}-(2s+1)\LR{u\ds u\HT D^{s}\ds w,D^{s}w}|\\
&\le|\LR{P_{s}^{(7)}(u,u,w)+P_{s}^{(7)}(u,w,v)+P_{s}^{(7)}(w,v,v),D^{s}w}|+|\LR{\HT\ds(u\ds u),(D^{s}w)^{2}}|/2\\
&\quad+|\LR{D^{s}\ds v\HT(u\ds w),D^{s}w}|+|\LR{D^{s}\ds v\HT(w\ds v),D^{s}w}|+|\LR{\ds(u^{2}\HT D^{s}\ds w),D^{s}w}|\\
&\quad+|\LR{uw\HT D^{s}\ds^{2}v,D^{s}w}|+|\LR{vw\HT D^{s}\ds^{2}v,D^{s}w}|+(2s+3)|\LR{\ds(zw)\HT D^{s}\ds v,D^{s}w}|/2\\
&\lesssim I_{s_{0}}(u,v)^{2}\{\|w\|_{H^{s}}^{2}+\|w\|_{H^{s_{0}-2}}^{2}\|v\|_{H^{s+2}}^{2}+(\|u\|_{H^{s}}^{2}+\|v\|_{H^{s}}^{2})\|w\|_{H^{s_{0}}}^{2}\},
}
which completes the proof.
\end{proof}

\begin{lem}\label{lem2.10new}
Let $s_{0}>7/2$ and $s\ge0$.
Let $u,v\in H^{\max\{s+2,s_{0}\}}(\T)$.
Then
\EQQS{
&|\LR{D^{s}\ds(u^{2}\HT\ds u-v^{2}\HT\ds v),D^{s}w}-2s\LR{u\ds u\HT D^{s}w,D^{s}w}|\\
&\lesssim I_{s_{0}}(u,v)^{2}\{\|w\|_{H^{s}}^{2}+\|w\|_{H^{s_{0}-2}}^{2}\|v\|_{H^{s+2}}^{2}+(\|u\|_{H^{s}}^{2}+\|v\|_{H^{s}}^{2})\|w\|_{H^{s_{0}}}^{2}\},
}
where $w=u-v$.
\end{lem}

\begin{proof}
Set $z=u+v$.
Lemma \ref{reduction} and \ref{comm.est.5} show that
\EQQS{
&|\LR{D^{s}\ds(u^{2}\HT\ds w+zw\HT\ds v),D^{s}w}-2s\LR{u\ds u\HT D^{s}w,D^{s}w}|\\
&\le|\LR{P_{s}^{(6)}(u,u,\HT w)+P_{s}^{(6)}(z,w,\HT v),D^{s}w}+|\LR{\ds(u\HT\ds u),(D^{s}w)^{2}}|\\
&\quad+2|\LR{u(\HT\ds w)D^{s}\ds v,D^{s}w}|+|\LR{\ds(u^{2}\HT D^{s}\ds w),D^{s}w}|\\
&\quad+|\LR{\ds(w\HT\ds v),(D^{s}w)^{2}}|/2+2|\LR{w(\HT\ds v)D^{s}\ds v,D^{s}w}|\\
&\quad+|\LR{zw\HT D^{s}\ds^{2}v,D^{s}w}|+(s+1)|\LR{\ds(zw)\HT D^{s}\ds v,D^{s}w}|\\
&\lesssim I_{s_{0}}(u,v)^{2}\{\|w\|_{H^{s}}^{2}+\|w\|_{H^{s_{0}-2}}^{2}\|v\|_{H^{s+2}}^{2}+(\|u\|_{H^{s}}^{2}+\|v\|_{H^{s}}^{2})\|w\|_{H^{s_{0}}}^{2}\},
}
which completes the proof.
\end{proof}

\begin{lem}\label{lem.fo}
Let $s_{0}>7/2$ and $s\ge0$.
Let $u,v\in H^{\max\{s+1,s_{0}\}}(\T)$.
Then
\EQQS{
&|\LR{D^{s}\ds(u^{4}-v^{4}),D^{s}w}|\\
&\lesssim I_{s_{0}}(u,v)^{3}\{\|w\|_{H^{s}}^{2}+\|w\|_{H^{s_{0}-1}}^{2}\|v\|_{H^{s+1}}^{2}+(\|u\|_{H^{s}}^{2}+\|v\|_{H^{s}}^{2})\|w\|_{H^{s_{0}}}^{2}\},
}
where $w=u-v$.
\end{lem}

\begin{proof}
Lemma \ref{Leibniz} and \ref{Leibniz2} show that
\EQQS{
&|\LR{D^{s}(u^{3}\ds w+w(u^{2}+uv+v^{2})\ds v),D^{s}w}|\\
&\le|\LR{[D^{s},u^{3}]\ds w,D^{s}w}|+|\LR{[D^{s},w(u^{2}+uv+v^{2})]\ds v,D^{s}w}|\\
&\quad+|\LR{\ds(u^{3}),(D^{s}w)^{2}}|/2+|\LR{w(u^{2}+uv+v^{2})D^{s}\ds v,D^{s}w}|\\
&\lesssim I_{s_{0}}(u,v)^{3}\{\|w\|_{H^{s}}^{2}+\|w\|_{H^{s_{0}-1}}^{2}\|v\|_{H^{s+1}}^{2}+(\|u\|_{H^{s}}^{2}+\|v\|_{H^{s}}^{2})\|w\|_{H^{s_{0}}}^{2}\},
}
which completes the proof.
\end{proof}

\begin{rem}
In Lemma \ref{lem2.3new}, \ref{lem2.4new}, \ref{lem2.5new}, \ref{lem2.6new}, \ref{lem2.8new}, \ref{lem2.9new}, \ref{lem2.10new} and \ref{lem.fo} with $s=0$, we do not have terms such as $\|w\|_{H^{s_{0}-j}}^{2}\|v\|_{H^{s+j}}^{2}$ for $j=1,2,3$ and $(\|u\|_{H^{s}}^{2}+\|v\|_{H^{s}}^{2})\|w\|_{H^{s_{0}}}^{2}$ in the right hand side.
This can be verified by a simple caluculation.
Indeed, for example, on Lemma \ref{lem2.3new} with $s=0$, we have
\EQQS{
\LR{\ds(u\ds^{2}u-v\ds^{2}v),w}
&=-\LR{u\ds^{2}w,\ds w}-\LR{w\ds^{2}v,\ds w}\\
&=\frac{1}{2}\LR{\ds u,(\ds w)^{2}}+\frac{1}{2}\LR{\ds^{3}v,w^{2}}.
}
The second term in the right hand side can be estimated by $\lesssim\|v\|_{H^{s_{0}}}\|w\|^{2}$.
For this reason, we obtain the following.
\end{rem}

\begin{lem}\label{NT0}
Let $s_{0}>7/2$ and $u,v\in H^{s_{0}}(\T)$.
Then
\EQQS{
&\left|\sum_{j=2}^{4}\LR{\ds(F_{j}(u)-F_{j}(v)),w}+\la_{1}(0)\LR{\ds u,(\ds w)^{2}}+\la_{2}(0)\LR{(\HT\ds^{2}u)\HT\ds w,w}\right.\\
&\quad+\la_{3}(0)\LR{u\ds u\HT\ds w,w}|\\
&\lesssim I_{s_{0}}(u,v)^{3}\|w\|^{2},
}
where $w=u-v$.
\end{lem}

Now, we estimate the time derivatives of $M_{s}^{(1)}$, $M_{s}^{(2)}$ and $M_{s}^{(3)}$.
The following lemma helps us to calculate correction terms.
Note that Lemma \ref{fourth} is more complicated than Lemma 2.8 in \cite{Tanaka} because of the presence of $\HT$.

\begin{lem}\label{fourth}
Let $f,g,h$ be sufficiently smooth real-valued functions defined on $\T$.
Then,
\EQQS{
&\LR{\HT\ds^{4}f,gh}+\LR{f\HT\ds^{4}g,h}+\LR{fg,\HT\ds^{4}h}\\
&=-\LR{[\HT,h]\ds^{4}f,g}-\LR{[\HT,f]\ds^{4}h,g}+4\LR{\ds^{3}f\HT g,\ds h}-4\LR{\ds f\HT\ds g,\ds^{2}h}+2\LR{\ds^{2}f\HT g,\ds^{2}h}.
}
\end{lem}

\begin{proof}
Observe that
\EQQS{
&\LR{fg,\HT\ds^{4}h}\\
&=-\LR{[\HT,f]\ds^{4}h,g}-\LR{\ds^{4}f\HT g+4\ds^{3}f\HT\ds g+6\ds^{2}f\HT\ds^{2}g+4\ds f\HT\ds^{3}g+f\HT\ds^{4}g,h}\\
&=-\LR{[\HT,f]\ds^{4}h,g}-\LR{\ds^{4}f\HT g,h}-4\LR{\ds f\HT\ds g,\ds^{2}h}+2\LR{\ds^{2}f\HT\ds^{2}g,h}-\LR{f\HT\ds^{4}g,h}\\
&=-\LR{[\HT,f]\ds^{4}h,g}+\LR{\ds^{4}f\HT g,h}+4\LR{\ds^{3}f\HT g,\ds h}-4\LR{\ds f\HT\ds g,\ds^{2}h}\\
&\quad+2\LR{\ds^{2}f\HT g,\ds^{2}h}-\LR{f\HT\ds^{4}g,h}.
}
Note that
\EQQS{
\LR{\HT\ds^{4}f,gh}+\LR{\ds^{4}f\HT g,h}=-\LR{[\HT,h]\ds^{4}f,g},
}
which completes the proof.
\end{proof}

\begin{lem}\label{m.e.1.2new}
Let $s_{0}>7/2$ and $s\ge1$.
Let $u,w\in H^{\max\{s+4,s_{0}\}}(\T)$.
Then
\EQQS{
&|\LR{(\HT\ds^{4}u)\HT D^{s}w,\HT D^{s-1} w}-\LR{uD^{s}\ds^{4}w,\HT D^{s-1}w}+\LR{u\HT D^{s}w,\HT D^{s}\ds^{3}w}
 -4\LR{\ds u,(D^{s}\ds w)^{2}}|\\
&\lesssim\|u\|_{H^{s_{0}}}\|w\|_{H^{s}}^{2}.
}
\end{lem}

\begin{proof}
We use Lemma \ref{fourth} with $f=u$, $g=\HT D^{s}w$ and $h=\HT D^{s-1}w$.
Then Lemma \ref{comm.est.H} shows that
\EQS{\label{eq3.6}
|\LR{[\HT,h]\ds^{4}f,g}|+|\LR{[\HT,f]\ds^{4}h,g}|+|\LR{\ds^{3}f\HT g,\ds h}|\lesssim\|u\|_{H^{s_{0}}}\|w\|_{H^{s}}^{2}.
}
Note that $-4\LR{\ds f\HT\ds g,\ds^{2}h}=4\LR{\ds u,(D^{s}\ds w)^{2}}$.
And finally, we see from the integration by parts that
\EQQS{
2|\LR{\ds^{2}f\HT g,\ds^{2}h}|
=|\LR{\ds^{3}u,(D^{s}w)^{2}}|\lesssim\|u\|_{H^{s_{0}}}\|w\|_{H^{s}}^{2}
}
which concludes the proof.
\end{proof}


\begin{lem}\label{lem2.14new}
Let $s_{0}>7/2$ and $s\ge1$.
Let $u,w\in H^{\max\{s+4,s_{0}\}}(\T)$.
Then
\EQQS{
&|\LR{\ds^{5}u,(D^{s-1}w)^{2}}-2\LR{(\HT\ds u)D^{s}\ds^{3}w,D^{s-1}w}
+4\LR{(\HT\ds^{2}u)\HT D^{s}\ds w,D^{s}w}|\\
&\lesssim\|u\|_{H^{s_{0}}}\|w\|_{H^{s}}^{2}.
}
\end{lem}

\begin{proof}
The integration by parts shows that
\EQQS{
&|\LR{\ds^{5}u,(D^{s-1}w)^{2}}-2\LR{(\HT\ds u)D^{s}\ds^{3}w,D^{s-1}w}
+4\LR{(\HT\ds^{2}u)\HT D^{s}\ds w,D^{s}w}|\\
&=|\LR{\ds^{4}uD^{s-1}\ds w,D^{s-1}w}-\LR{(\HT u)D^{s}\ds^{4}w,D^{s-1}w}+\LR{(\HT u)D^{s}\ds^{3}w,D^{s-1}\ds w}\\
&\quad-2\LR{(\HT\ds^{2}u)\HT D^{s}\ds w,D^{s}w}|,
}
which allows us to use Lemma \ref{fourth} with $f=\HT u$, $g=D^{s-1}\ds w$ and $h=D^{s-1}w$.
It is cleat that \eqref{eq3.6} holds in this case.
Lemma $\ref{reduction}$ implies that
\[|\LR{\ds f\HT\ds g,\ds^{2}h}|=|\LR{\ds((\HT\ds u)\HT D^{s}\ds w),D^{s}w}|\lesssim\|u\|_{H^{s_{0}}}\|w\|_{H^{s}}^{2}.\]
On the other hand, we have
$2\LR{\ds^{2}f\HT g,\ds^{2}h}=-2\LR{(\HT\ds^{2}u)\HT D^{s}\ds w,D^{s}w}$,
which completes the proof.
\end{proof}


\begin{lem}\label{lem2.21new}
Let $s_{0}>7/2$ and $s\ge1$.
Let $u,w\in H^{\max\{s+3,s_{0}\}}(\T)$.
Then
\EQQS{
|\LR{u\HT\ds^{4}u,(D^{s-1}w)^{2}}+\LR{u^{2}D^{s-1}w,D^{s}\ds^{3}w}-4\LR{u\ds u\HT D^{s}\ds w,D^{s}w}|
\lesssim\|u\|_{H^{s_{0}}}^{2}\|w\|_{H^{s}}^{2}.
}
\end{lem}

\begin{proof}
Adding and subtraction a term, we have
\EQQS{
&2|\LR{u\HT\ds^{4}u,(D^{s-1}w)^{2}}+\LR{u^{2}D^{s-1}w,D^{s}\ds^{3}w}-4\LR{u\ds u\HT D^{s}\ds w,D^{s}w}|\\
&\le|\LR{\HT\ds^{4}(u^{2}),(D^{s-1}w)^{2}}+2\LR{u^{2}D^{s-1}w,D^{s}\ds^{3}w}-8\LR{u\ds u\HT D^{s}\ds w,D^{s}w}|\\
&\quad+|\LR{2u\HT\ds^{4}u-\HT\ds^{4}(u^{2}),(D^{s-1}w)^{2}}|.
}
Lemma \ref{comm.est.H} shows that the second term in the right hand side can be estimated by $\lesssim\|u\|_{H^{s_{0}}}^{2}\|w\|_{H^{s}}^{2}$.
We use Lemma \ref{fourth} with $f=u^{2}$, $g=h=D^{s-1}w$.
It is clear that \eqref{eq3.6} holds in this case.
Note that $-4\LR{\ds f\HT\ds g,\ds^{2}h}=8\LR{u\ds u\HT D^{s}\ds w,D^{s}w}$.
Finally, we have
\[|\LR{\ds^{2}f\HT g,\ds^{2}h}|=|\LR{\ds(\ds^{2}(u^{2})\HT D^{s-1}w),\HT D^{s}w}|\lesssim\|u\|_{H^{s_{0}}}^{2}\|w\|_{H^{s}}^{2},\]
which completes the proof.
\end{proof}


We observe the first order derivative loss resulting from $M_{s}^{(1)}$.

\begin{lem}\label{lem4.23}
Let $s_{0}>7/2$ and $s\ge1$.
Let $u,v\in H^{\max\{s+3,s_{0}\}}(\T)$.
Then
\EQQS{
&|\LR{u\HT D^{s}\ds(u\ds^{2}u-v\ds^{2}v),\HT D^{s-1}w}+(s-3)\LR{u\ds u\HT D^{s}\ds w,D^{s}w}|\\
&\lesssim I_{s_{0}}(u,v)^{2}\{\|w\|_{H^{s}}^{2}+\|w\|_{H^{s_{0}-2}}^{2}\|v\|_{H^{s+2}}^{2}+(\|u\|_{H^{s}}^{2}+\|v\|_{H^{s}}^{2})\|w\|_{H^{s_{0}}}^{2}\},
}
where $w=u-v$.
\end{lem}

\begin{proof}
Lemma \ref{comm.est.1} and \ref{reduction} show that
\EQQS{
&|\LR{u\HT D^{s}\ds(u\ds^{2}w+w\ds^{2}v),\HT D^{s-1}w}+(s-3)\LR{u\ds u\HT D^{s}\ds w,D^{s}w}|\\
&\le|\LR{u(P_{s}^{(2)}(u,w)+P_{s}^{(2)}(w,v)),\HT D^{s-1}w}|+|\LR{\ds(u\ds^{2}w\HT D^{s-1}w),\HT D^{s}u}|\\
&\quad+|\LR{\ds^{3}(u^{2})\HT D^{s}w,\HT D^{s-1}w}|+|\LR{\ds^{2}(u^{2})\HT D^{s}w,D^{s}w}|+|\LR{\ds(u^{2}\HT D^{s}\ds w),D^{s}w}|\\
&\quad+(s+1)|\LR{\ds(\ds^{2}(u^{2})\HT D^{s-1}w),\HT D^{s}w}|/2+s(s+1)|\LR{\ds(u\ds^{2}u\HT D^{s-1}w),\HT D^{s}w}|/2\\
&\quad+|\LR{\ds(u\ds^{2}v\HT D^{s-1}w),\HT D^{s}w}|+|\LR{\ds(uw\HT D^{s-1}w),\HT D^{s}\ds^{2}v}|\\
&\quad+(s+1)|\LR{\ds(u\ds w\HT D^{s-1}w),\HT D^{s}\ds v}|+s(s+1)|\LR{\ds(u\ds^{2}w\HT D^{s-1}w),\HT D^{s}v}|/2\\
&\lesssim I_{s_{0}}(u,v)^{2}\{\|w\|_{H^{s}}^{2}+\|w\|_{H^{s_{0}-2}}^{2}\|v\|_{H^{s+2}}^{2}+(\|u\|_{H^{s}}^{2}+\|v\|_{H^{s}}^{2})\|w\|_{H^{s_{0}}}^{2}\},
}
which completes the proof.
\end{proof}

\begin{lem}\label{lem4.25}
Let $s_{0}>7/2$ and $s\ge1$.
Let $u,v\in H^{\max\{s+2,s_{0}\}}(\T)$.
Then
\EQQS{
&|\LR{u\HT D^{s}w,D^{s}(u\ds^{2}u-v\ds^{2}v)}+(s-2)\LR{u\ds u\HT D^{s}\ds w,D^{s}w}|\\
&\lesssim I_{s_{0}}(u,v)^{2}\{\|w\|_{H^{s}}^{2}+\|w\|_{H^{s_{0}-2}}^{2}\|v\|_{H^{s+2}}^{2}+(\|u\|_{H^{s}}^{2}+\|v\|_{H^{s}}^{2})\|w\|_{H^{s_{0}}}^{2}\},
}
where $w=u-v$.
\end{lem}

\begin{proof}
Lemma \ref{comm.est.7} together with Lemma \ref{reduction} shows that
\EQQS{
&|\LR{u\HT D^{s}w,D^{s}(u\ds^{2}w+w\ds^{2}v)}+(s-2)\LR{u\ds u\HT D^{s}\ds w,D^{s}w}|\\
&\le|\LR{u\HT D^{s}w,D^{s}(u\ds^{2}w)-uD^{s}\ds^{2}w-s\ds uD^{s}\ds w}|\\
&\quad+|\LR{u\HT D^{s}w,D^{s}(w\ds^{2}v)-wD^{s}\ds^{2}v-s\ds wD^{s}\ds v}|+|\LR{\ds^{2}(u^{2})\HT D^{s}w,D^{s}w}|\\
&\quad+|\LR{\ds(u^{2}\HT D^{s}\ds w),D^{s}w}|+s|\LR{\ds(u\ds u)\HT D^{s}w,D^{s}w}|\\
&\quad+|\LR{u\HT D^{s}w,wD^{s}\ds^{2}v}|+s|\LR{u\HT D^{s}w,\ds wD^{s}\ds v}|\\
&\lesssim I_{s_{0}}(u,v)^{2}\{\|w\|_{H^{s}}^{2}+\|w\|_{H^{s_{0}-2}}^{2}\|v\|_{H^{s+2}}^{2}+(\|u\|_{H^{s}}^{2}+\|v\|_{H^{s}}^{2})\|w\|_{H^{s_{0}}}^{2}\},
}
which completes the proof.
\end{proof}

\begin{lem}\label{lem4.26}
Let $s_{0}>7/2$ and $s\ge1$.
Let $u,v\in H^{\max\{s+2,s_{0}\}}(\T)$.
Then
\EQQS{
&|\LR{u\HT D^{s}\ds((\ds u)^{2}-(\ds v)^{2}),\HT D^{s-1}w}+2\LR{u\ds u\HT D^{s}\ds w,D^{s}w}|\\
&\lesssim I_{s_{0}}(u,v)^{2}\{\|w\|_{H^{s}}^{2}+\|w\|_{H^{s_{0}-1}}^{2}\|v\|_{H^{s+1}}^{2}+(\|u\|_{H^{s}}^{2}+\|v\|_{H^{s}}^{2})\|w\|_{H^{s_{0}}}^{2}\},
}
where $w=u-v$.
\end{lem}

\begin{proof}
Set $z=u+v$.
Lemma \ref{comm.est.7} shows that
\EQQS{
&|\LR{u\HT D^{s}\ds(\ds z\ds w),\HT D^{s-1}w}+2\LR{u\ds u\HT D^{s}\ds w,D^{s}w}|\\
&\le|\LR{uP_{s}^{(4)}(z,w),\HT D^{s-1}w}|+2|\LR{\ds(u\ds w\HT D^{s-1}w),\HT D^{s}\ds v}|\\
&\quad+(s+1)|\LR{\ds(u\ds^{2}w\HT D^{s-1}w),\HT D^{s}z}|+(s+1)|\LR{\ds(u\ds^{2}z\HT D^{s-1}w),\HT D^{s}w}|\\
&\lesssim I_{s_{0}}(u,v)^{2}\{\|w\|_{H^{s}}^{2}+\|w\|_{H^{s_{0}-1}}^{2}\|v\|_{H^{s+1}}^{2}+(\|u\|_{H^{s}}^{2}+\|v\|_{H^{s}}^{2})\|w\|_{H^{s_{0}}}^{2}\},
}
which completes the proof.
\end{proof}

\begin{lem}\label{lem4.28}
Let $s_{0}>7/2$ and $s\ge1$.
Let $u,v\in H^{\max\{s+1,s_{0}\}}(\T)$.
Then
\EQQS{
&|\LR{u\HT D^{s}w,D^{s}((\ds u)^{2}-(\ds v)^{2})}+2\LR{u\ds u\HT D^{s}\ds w,D^{s}w}|\\
&\lesssim I_{s_{0}}(u,v)^{2}\{\|w\|_{H^{s}}^{2}+\|w\|_{H^{s_{0}-1}}^{2}\|v\|_{H^{s+1}}^{2}+(\|u\|_{H^{s}}^{2}+\|v\|_{H^{s}}^{2})\|w\|_{H^{s_{0}}}^{2}\},
}
where $w=u-v$.
\end{lem}

\begin{proof}
Lemma \ref{comm.est.3} shows that
\EQQS{
&|\LR{u\HT D^{s}w,D^{s}(\ds z\ds w)}+2\LR{u\ds u\HT D^{s}\ds w,D^{s}w}|\\
&\le|\LR{u\HT D^{s}w,D^{s}(\ds z\ds w)-D^{s}\ds z\ds w-\ds zD^{s}\ds w}|+2|\LR{\ds(u\ds u)\HT D^{s}w,D^{s}w}|\\
&\quad+2|\LR{u\ds w\HT D^{s}w,D^{s}\ds v}|\\
&\lesssim I_{s_{0}}(u,v)^{2}\{\|w\|_{H^{s}}^{2}+\|w\|_{H^{s_{0}-1}}^{2}\|v\|_{H^{s+1}}^{2}+(\|u\|_{H^{s}}^{2}+\|v\|_{H^{s}}^{2})\|w\|_{H^{s_{0}}}^{2}\},
}
which completes the proof.
\end{proof}

\begin{lem}\label{lem4.29}
Let $s_{0}>7/2$ and $s\ge1$.
Let $u,v\in H^{\max\{s+2,s_{0}\}}(\T)$.
Then
\EQQS{
&|\LR{u\HT D^{s}\ds((\HT\ds u)^{2}-(\HT\ds v)^{2}),\HT D^{s-1}w}|\\
&\lesssim I_{s_{0}}(u,v)^{2}\{\|w\|_{H^{s}}^{2}+\|w\|_{H^{s_{0}-1}}^{2}\|v\|_{H^{s+1}}^{2}+(\|u\|_{H^{s}}^{2}+\|v\|_{H^{s}}^{2})\|w\|_{H^{s_{0}}}^{2}\},
}
where $w=u-v$.
\end{lem}

\begin{proof}
Set $z=u+v$.
Lemma \ref{comm.est.2} shows that
\EQQS{
&|\LR{u\HT D^{s}\ds((\HT\ds z)(\HT\ds w)),\HT D^{s-1}w}|\\
&\le|\LR{uP_{s}^{(4)}(\HT z,\HT w),\HT D^{s-1}w}+3|\LR{\ds(u\HT\ds u),(D^{s}w)^{2}}|\\
&\quad+2|\LR{\ds^{2}(u\HT\ds u)D^{s}w,\HT D^{s-1}w}|+2|\LR{\ds(u(\HT\ds w)\HT D^{s-1}w),D^{s}\ds v}|\\
&\quad+(s+1)|\LR{\ds(u(\HT\ds^{2}w)\HT D^{s-1}w),D^{s}z)}|+(s+1)|\LR{\ds(u(\HT\ds^{2}z)\HT D^{s-1}w),D^{s}w)}|\\
&\lesssim I_{s_{0}}(u,v)^{2}\{\|w\|_{H^{s}}^{2}+\|w\|_{H^{s_{0}-1}}^{2}\|v\|_{H^{s+1}}^{2}+(\|u\|_{H^{s}}^{2}+\|v\|_{H^{s}}^{2})\|w\|_{H^{s_{0}}}^{2}\},
}
which completes the proof.
\end{proof}

\begin{lem}\label{lem4.30}
Let $s_{0}>7/2$ and $s\ge1$.
Let $u,v\in H^{\max\{s+1,s_{0}\}}(\T)$.
Then
\EQQS{
&|\LR{u\HT D^{s}w,D^{s}((\HT\ds u)^{2}-(\HT\ds v)^{2})}|\\
&\lesssim I_{s_{0}}(u,v)^{2}\{\|w\|_{H^{s}}^{2}+\|w\|_{H^{s_{0}-1}}^{2}\|v\|_{H^{s+1}}^{2}+(\|u\|_{H^{s}}^{2}+\|v\|_{H^{s}}^{2})\|w\|_{H^{s_{0}}}^{2}\},
}
where $w=u-v$.
\end{lem}

\begin{proof}
Set $z=u+v$.
Lemma \ref{comm.est.3} shows that
\EQQS{
&|\LR{u\HT D^{s}w,D^{s}((\HT\ds z)\HT\ds w)}|\\
&\le|\LR{u\HT D^{s}w,D^{s}((\HT\ds z)\HT\ds w)-(\HT\ds w)\HT D^{s}\ds z-(\HT\ds z)\HT D^{s}\ds w}|\\
&\quad+|\LR{\ds(u\HT\ds^{2}u),(\HT D^{s}w)^{2}}|+2|\LR{u(\HT\ds w)\HT D^{s}w,\HT D^{s}\ds v}|\\
&\lesssim I_{s_{0}}(u,v)^{2}\{\|w\|_{H^{s}}^{2}+\|w\|_{H^{s_{0}-1}}^{2}\|v\|_{H^{s+1}}^{2}+(\|u\|_{H^{s}}^{2}+\|v\|_{H^{s}}^{2})\|w\|_{H^{s_{0}}}^{2}\},
}
which completes the proof.
\end{proof}

\begin{lem}\label{lem4.31}
Let $s_{0}>7/2$ and $s\ge1$.
Let $u,v\in H^{\max\{s+3,s_{0}\}}(\T)$.
Then
\EQQS{
&|\LR{uD^{s}\ds(u\HT\ds^{2}u-v\HT\ds^{2}v),\HT D^{s-1}w}+(s-3)\LR{u\ds u\HT D^{s}\ds w,D^{s}w}|\\
&\lesssim I_{s_{0}}(u,v)^{2}\{\|w\|_{H^{s}}^{2}+\|w\|_{H^{s_{0}-2}}^{2}\|v\|_{H^{s+2}}^{2}+(\|u\|_{H^{s}}^{2}+\|v\|_{H^{s}}^{2})\|w\|_{H^{s_{0}}}^{2}\},
}
where $w=u-v$.
\end{lem}

\begin{proof}
We see from Lemma \ref{comm.est.1} and \ref{reduction} that
\EQQS{
&|\LR{uD^{s}\ds(u\HT\ds^{2}w+w\HT\ds^{2}v),\HT D^{s-1}w}+(s-3)\LR{u\ds u\HT D^{s}\ds w,D^{s}w}|\\
&\le|\LR{u(P_{s}^{(1)}(u,\HT w)+P_{s}^{(1)}(w,\HT v)),\HT D^{s-1}w}|+|\LR{\ds(u(\HT\ds^{2}w)\HT D^{s-1}w),D^{s}u}|\\
&\quad+|\LR{\ds(\ds^{2}(u^{2})\HT D^{s-1}w),\HT D^{s}w}|+|\LR{\ds(u^{2}\HT D^{s}\ds w),D^{s}w}|\\
&\quad+(s+1)|\LR{\ds(\ds(u\ds u)\HT D^{s-1}w),\HT D^{s}w}|+s(s+1)|\LR{\ds(u\ds^{2}u\HT D^{s-1}w),\HT D^{s}w}|/2\\
&\quad+|\LR{\ds(u(\HT\ds^{2}v)\HT D^{s-1}w),D^{s}w}|+|\LR{\ds(uw\HT D^{s-1}w),\HT D^{s}\ds^{2}v}|\\
&\quad+(s+1)|\LR{\ds(u\ds w\HT D^{s-1}w),\HT D^{s}\ds v}|+s(s+1)|\LR{\ds(u\ds^{2}w\HT D^{s-1}w),\HT D^{s}v}|/2\\
&\lesssim I_{s_{0}}(u,v)^{2}\{\|w\|_{H^{s}}^{2}+\|w\|_{H^{s_{0}-2}}^{2}\|v\|_{H^{s+2}}^{2}+(\|u\|_{H^{s}}^{2}+\|v\|_{H^{s}}^{2})\|w\|_{H^{s_{0}}}^{2}\},
}
which completes the proof.
\end{proof}

\begin{lem}\label{lem4.32}
Let $s_{0}>7/2$ and $s\ge1$.
Let $u,v\in H^{\max\{s+2,s_{0}\}}(\T)$.
Then
\EQQS{
&|\LR{u\HT D^{s}w,D^{s-1}\ds(u\HT\ds^{2}u-v\HT\ds^{2}v)}+(s-2)\LR{u\ds u\HT D^{s}\ds w,D^{s}w}|\\
&\lesssim I_{s_{0}}(u,v)^{2}\{\|w\|_{H^{s}}^{2}+\|w\|_{H^{s_{0}-1}}^{2}\|v\|_{H^{s+1}}^{2}+(\|u\|_{H^{s}}^{2}+\|v\|_{H^{s}}^{2})\|w\|_{H^{s_{0}}}^{2}\},
}
where $w=u-v$.
\end{lem}

\begin{proof}
Lemma \ref{comm.est.7} and \ref{reduction} shows that
\EQQS{
&|\LR{u\HT D^{s}w,D^{s-1}\ds(u\HT\ds^{2}w+w\HT\ds^{2}v)}+(s-2)\LR{u\ds u\HT D^{s}\ds w,D^{s}w}|\\
&\le|\LR{u\HT D^{s}w,D^{s-1}\ds(u\HT\ds^{2}w)-u\HT D^{s-1}\ds^{3}w-s\ds u\HT D^{s-1}\ds^{2}w}|\\
&\quad+|\LR{u\HT D^{s}w,D^{s-1}\ds(w\HT\ds^{2}v)-w\HT D^{s-1}\ds^{3}v-s\ds w\HT D^{s-1}\ds^{2}v}|\\
&\quad+|\LR{[\HT,u^{2}]D^{s}\ds^{2}w,D^{s}w}|+|\LR{\ds(u^{2}\HT D^{s}\ds w),D^{s}w}|\\
&\quad+s|\LR{[\HT,u\ds u]D^{s}\ds w,D^{s}w}|+|\LR{u\HT D^{s}w,wD^{s+2}v}|+s|\LR{u\HT D^{s}w,\ds wD^{s}\ds v}|\\
&\lesssim I_{s_{0}}(u,v)^{2}\{\|w\|_{H^{s}}^{2}+\|w\|_{H^{s_{0}-1}}^{2}\|v\|_{H^{s+1}}^{2}+(\|u\|_{H^{s}}^{2}+\|v\|_{H^{s}}^{2})\|w\|_{H^{s_{0}}}^{2}\},
}
which completes the proof.
\end{proof}

\begin{lem}\label{lem4.34}
Let $s_{0}>7/2$ and $s\ge1$.
Let $u,v\in H^{\max\{s+2,s_{0}\}}(\T)$.
Then
\EQQS{
&|\LR{(\HT\ds u)D^{s-1}w,D^{s-1}\ds(u\ds^{2}u-v\ds^{2}v)}|\\
&\lesssim I_{s_{0}}(u,v)^{2}\{\|w\|_{H^{s}}^{2}+\|w\|_{H^{s_{0}-1}}^{2}\|v\|_{H^{s+1}}^{2}+(\|u\|_{H^{s}}^{2}+\|v\|_{H^{s}}^{2})\|w\|_{H^{s_{0}}}^{2}\},
}
where $w=u-v$.
\end{lem}

\begin{proof}
Note that
\EQQS{
&\LR{(\HT\ds u)D^{s-1}w,uD^{s-1}\ds^{3}w}\\
&=-\LR{\ds(u\HT\ds u)D^{s-1}w,D^{s-1}\ds^{2}w}-\LR{u(\HT\ds u)D^{s-1}\ds w,D^{s-1}\ds^{2}w}\\
&=\LR{\ds(\ds(u\HT\ds u)D^{s-1}w),D^{s-1}\ds w}+\frac{1}{2}\LR{\ds(u\HT\ds u),(D^{s-1}\ds w)^{2}}.
}
Lemma \ref{comm.est.7} shows that
\EQQS{
&|\LR{(\HT\ds u)D^{s-1}w,D^{s-1}\ds(u\ds^{2}w+w\ds^{2}v)}|\\
&\le|\LR{(\HT\ds u)D^{s-1}w,D^{s-1}\ds(u\ds^{2}w)-uD^{s}\ds^{3}w-s\ds uD^{s-1}\ds^{2}w}|\\
&\quad+|\LR{(\HT\ds u)D^{s-1}w,D^{s-1}\ds(w\ds^{2}v)-wD^{s-1}\ds^{3}v-s\ds wD^{s-1}\ds^{2}v}|\\
&\quad+|\LR{(\HT\ds u)D^{s-1}w,uD^{s-1}\ds^{3}w}|+s|\LR{\ds(\ds u(\HT\ds u)D^{s-1}w),D^{s-1}\ds w}|\\
&\quad+|\LR{\ds(w(\HT\ds u)D^{s-1}w),D^{s-1}\ds^{2}v}|+s|\LR{\ds(\ds w(\HT\ds u)D^{s-1}w),D^{s-1}\ds v}|\\
&\lesssim I_{s_{0}}(u,v)^{2}\{\|w\|_{H^{s}}^{2}+\|w\|_{H^{s_{0}-1}}^{2}\|v\|_{H^{s+1}}^{2}+(\|u\|_{H^{s}}^{2}+\|v\|_{H^{s}}^{2})\|w\|_{H^{s_{0}}}^{2}\},
}
which completes the proof.
\end{proof}

\begin{lem}\label{lem4.35}
Let $s_{0}>7/2$ and $s\ge1$.
Let $u,v\in H^{\max\{s+1,s_{0}\}}(\T)$.
Then
\EQQS{
&|\LR{(\HT\ds u)D^{s-1}w,D^{s-1}\ds((\ds u)^{2}-(\ds v)^{2})}|\\
&\lesssim I_{s_{0}}(u,v)^{2}\{\|w\|_{H^{s}}^{2}+(\|u\|_{H^{s}}^{2}+\|v\|_{H^{s}}^{2})\|w\|_{H^{s_{0}}}^{2}\},
}
where $w=u-v$.
\end{lem}

\begin{proof}
Lemma \ref{comm.est.3} shows that
\EQQS{
&|\LR{(\HT\ds u)D^{s-1}w,D^{s-1}\ds(\ds z\ds w)}|\\
&\le|\LR{(\HT\ds u)D^{s-1}w,D^{s-1}\ds(\ds z\ds w)-D^{s-1}\ds^{2}z\ds w-\ds zD^{s-1}\ds^{2}w}|\\
&\quad+|\LR{\ds(\ds w(\HT\ds u)D^{s-1}w),D^{s-1}\ds z}|+|\LR{\ds(\ds z(\HT\ds u)D^{s-1}w),D^{s-1}\ds w}|\\
&\lesssim I_{s_{0}}(u,v)^{2}\{\|w\|_{H^{s}}^{2}+(\|u\|_{H^{s}}^{2}+\|v\|_{H^{s}}^{2})\|w\|_{H^{s_{0}}}^{2}\},
}
which completes the proof.
\end{proof}

\begin{lem}\label{lem4.37}
Let $s_{0}>7/2$ and $s\ge1$.
Let $u,v\in H^{\max\{s+2,s_{0}\}}(\T)$.
Then
\EQQS{
&|\LR{(\HT\ds u)D^{s-1}w,D^{s}(u\HT\ds^{2}u-v\HT\ds^{2}v)}|\\
&\lesssim I_{s_{0}}(u,v)^{2}\{\|w\|_{H^{s}}^{2}+\|w\|_{H^{s_{0}-1}}^{2}\|v\|_{H^{s+1}}^{2}+(\|u\|_{H^{s}}^{2}+\|v\|_{H^{s}}^{2})\|w\|_{H^{s_{0}}}^{2}\},
}
where $w=u-v$.
\end{lem}

\begin{proof}
Note that
\EQQS{
&\LR{(\HT\ds u)D^{s-1}w,u\HT D^{s}\ds^{2}w}\\
&=-\LR{\ds(u\HT\ds u)D^{s-1}w,\HT D^{s}\ds w}+\LR{u(\HT\ds u)\HT D^{s}w,\HT D^{s}\ds w}\\
&=\LR{\ds(\ds(u\HT\ds u)D^{s-1}w),\HT D^{s}w}-\frac{1}{2}\LR{\ds(u\HT\ds u),(\HT D^{s}w)^{2}}.
}
We have
\EQQS{
&|\LR{(\HT\ds u)D^{s-1}w,D^{s}(u\HT\ds^{2}w+w\HT\ds^{2}v)}|\\
&\le|\LR{(\HT\ds u)D^{s-1}w,D^{s}(u\HT\ds^{2}w)-u\HT D^{s}\ds^{2}w-s\ds u\HT D^{s}\ds w}|\\
&\quad+|\LR{(\HT\ds u)D^{s-1}w,D^{s}(w\HT\ds^{2}v)-w\HT D^{s}\ds^{2}v-s\ds w\HT D^{s}\ds v}|\\
&\quad+|\LR{(\HT\ds u)D^{s-1}w,u\HT D^{s}\ds^{2}w}|+s|\LR{(\HT\ds u)D^{s-1}w,\ds u\HT D^{s}\ds w}|\\
&\quad+|\LR{(\HT\ds u)D^{s-1}w,w\HT D^{s}\ds^{2}v}|+s|\LR{(\HT\ds u)D^{s-1}w,\ds w\HT D^{s}\ds v}|
}
which completes the proof by Lemma \ref{comm.est.7}.
\end{proof}

\begin{lem}\label{lem4.39}
Let $s_{0}>7/2$ and $s\ge1$.
Let $u,v\in H^{\max\{s+2,s_{0}\}}(\T)$.
Then
\EQQS{
&|\LR{uD^{s}\ds(u^{2}\ds u-v^{2}\ds v),\HT D^{s-1}w}|\\
&\lesssim I_{s_{0}}(u,v)^{3}\{\|w\|_{H^{s}}^{2}+\|w\|_{H^{s_{0}-1}}^{2}\|v\|_{H^{s+1}}^{2}+(\|u\|_{H^{s}}^{2}+\|v\|_{H^{s}}^{2})\|w\|_{H^{s_{0}}}^{2}\},
}
where $w=u-v$.
\end{lem}

\begin{proof}
Set $z=u+v$.
Lemma \ref{Leibniz2} shows that
\EQQS{
&|\LR{uD^{s}\ds(u^{2}\ds w+zw\ds v),\HT D^{s-1}w}|\\
&\le|\LR{\ds uD^{s}(u^{2}\ds w+zw\ds v),\HT D^{s-1}w}|+|\LR{uD^{s}(u^{2}\ds w+zw\ds v),D^{s}w}|\\
&\le|\LR{D(\ds u\HT D^{s-1}w),D^{s-1}(u^{2}\ds w+zw\ds v)}|+|\LR{u[D^{s},u^{2}]\ds w,D^{s}w}|\\
&\quad+\frac{3}{2}|\LR{u^{2}\ds u,(D^{s}w)^{2}}|+|\LR{uD^{s}(zw\ds v),D^{s}w}|,
}
which completes the proof.
\end{proof}

\begin{lem}\label{lem4.41}
Let $s_{0}>7/2$ and $s\ge1$.
Let $u,v\in H^{\max\{s+2,s_{0}\}}(\T)$.
Then
\EQQS{
&|\LR{u\HT D^{s}\ds(u\HT(u\ds u)-v\HT(v\ds v)),\HT D^{s-1}w}|\\
&\lesssim I_{s_{0}}(u,v)^{3}\{\|w\|_{H^{s}}^{2}+\|w\|_{H^{s_{0}-1}}^{2}\|v\|_{H^{s+1}}^{2}+(\|u\|_{H^{s}}^{2}+\|v\|_{H^{s}}^{2})\|w\|_{H^{s_{0}}}^{2}\},
}
where $w=u-v$.
\end{lem}

\begin{proof}
Note that
\EQQS{
&|\LR{u\HT D^{s}\ds(u\HT(u\ds w)+u\HT(w\ds v)+w\HT(v\ds v)),\HT D^{s-1}w}|\\
&\le|\LR{D(\ds u\HT D^{s-1}w),\HT D^{s-1}(u\HT(u\ds w))}|+|\LR{u\HT D^{s}(u\HT(u\ds w)),D^{s}w}|\\
&\quad+|\LR{\ds(u\HT D^{s-1}w),\HT D^{s}(u\HT(w\ds v)+w\HT(v\ds v))}|.
}
The second term in the right hand side can be estimated as follows:
\EQQS{
&|\LR{u\HT D^{s}(u\HT(u\ds w)),D^{s}w}|\\
&\le|\LR{\HT D^{s}(u\HT(u\ds w))+u^{2}D^{s}\ds w,uD^{s}w}|+\frac{3}{2}|\LR{u^{2}\ds u,(D^{s}w)^{2}}|.
}
Then Lemma \ref{comm.est.9} completes the proof.
\end{proof}

\begin{lem}\label{lem4.42}
Let $s_{0}>7/2$ and $s\ge1$.
Let $u,v\in H^{\max\{s+2,s_{0}\}}(\T)$.
Then
\EQQS{
&|\LR{u\HT D^{s}\ds(u^{2}\HT\ds u-v^{2}\HT\ds v),\HT D^{s-1}w}|\\
&\lesssim I_{s_{0}}(u,v)^{3}\{\|w\|_{H^{s}}^{2}+\|w\|_{H^{s_{0}-1}}^{2}\|v\|_{H^{s+1}}^{2}+(\|u\|_{H^{s}}^{2}+\|v\|_{H^{s}}^{2})\|w\|_{H^{s_{0}}}^{2}\},
}
where $w=u-v$.
\end{lem}

\begin{proof}
Set $z=u+v$.
Note that
\EQQS{
&|\LR{u\HT D^{s}\ds(u^{2}\HT\ds w+zw\HT\ds v),\HT D^{s-1}w}|\\
&\le|\LR{D(\ds u\HT D^{s-1}w),\HT D^{s-1}(u^{2}\HT\ds w)}|+|\LR{u[\HT D^{s},u^{2}]\HT\ds w,D^{s}w}|\\
&\quad+\frac{3}{2}|\LR{u^{2}\ds u,(D^{s}w)^{2}}|+|\LR{\ds(u\HT D^{s-1}w),\HT D^{s}(zw\HT\ds v)}|.
}
Then Lemma \ref{Leibniz2} completes the proof.
\end{proof}

\begin{lem}\label{lem4.43}
Let $s_{0}>7/2$ and $s\ge1$.
Let $u,v\in H^{\max\{s+2,s_{0}\}}(\T)$.
Then
\EQQS{
&|\LR{u\HT D^{s}w,D^{s-1}\ds(u^{2}\ds u-v^{2}\ds v)}|\\
&\lesssim I_{s_{0}}(u,v)^{3}\{\|w\|_{H^{s}}^{2}+\|w\|_{H^{s_{0}-1}}^{2}\|v\|_{H^{s+1}}^{2}+(\|u\|_{H^{s}}^{2}+\|v\|_{H^{s}}^{2})\|w\|_{H^{s_{0}}}^{2}\},
}
where $w=u-v$.
\end{lem}

\begin{proof}
Note that
\EQQS{
&|\LR{u\HT D^{s}w,D^{s-1}\ds(u^{2}\ds w+zw\ds v)}|\\
&\le|\LR{u\HT D^{s}w,[D^{s-1}\ds,u^{2}]\ds w}|+\frac{3}{2}|\LR{u^{2}\ds u,(\HT D^{s}w)^{2}}|\\
&\quad+|\LR{u\HT D^{s}w,D^{s-1}\ds(zw\ds v)}|.
}
Then, Lemma \ref{Leibniz2} completes the proof.
\end{proof}

\begin{lem}\label{lem4.44}
Let $s_{0}>7/2$ and $s\ge1$.
Let $u,v\in H^{\max\{s+1,s_{0}\}}(\T)$.
Then
\EQQS{
&|\LR{u\HT D^{s}w,D^{s}(u\HT(u\ds u)-v\HT(v\ds v))}|\\
&\lesssim I_{s_{0}}(u,v)^{3}\{\|w\|_{H^{s}}^{2}+\|w\|_{H^{s_{0}-1}}^{2}\|v\|_{H^{s+1}}^{2}+(\|u\|_{H^{s}}^{2}+\|v\|_{H^{s}}^{2})\|w\|_{H^{s_{0}}}^{2}\},
}
where $w=u-v$.
\end{lem}

\begin{proof}
Note that
\EQQS{
&|\LR{u\HT D^{s}w,D^{s}(u\HT(u\ds w)+u\HT(w\ds v)+w\HT(v\ds v))}|\\
&\le|\LR{u\HT D^{s}w,D^{s}(u\HT(u\ds w))-u^{2}\HT D^{s}\ds w}|+\frac{3}{2}|\LR{u^{2}\ds u,(\HT D^{s}w)^{2}}|\\
&\quad+|\LR{u\HT D^{s}w,D^{s}(u\HT(w\ds v)+w\HT(v\ds v))}|.
}
We see that Lemma \ref{comm.est.9} completes the proof.
\end{proof}

\begin{lem}\label{lem4.45}
Let $s_{0}>7/2$ and $s\ge1$.
Let $u,v\in H^{\max\{s+1,s_{0}\}}(\T)$.
Then
\EQQS{
&|\LR{u\HT D^{s}w,D^{s}(u^{2}\HT\ds u-v^{2}\HT\ds v)}|\\
&\lesssim I_{s_{0}}(u,v)^{3}\{\|w\|_{H^{s}}^{2}+\|w\|_{H^{s_{0}-1}}^{2}\|v\|_{H^{s+1}}^{2}+(\|u\|_{H^{s}}^{2}+\|v\|_{H^{s}}^{2})\|w\|_{H^{s_{0}}}^{2}\},
}
where $w=u-v$.
\end{lem}

\begin{proof}
Set $z=u+v$.
Note that
\EQQS{
&|\LR{u\HT D^{s}w,D^{s}(u^{2}\HT\ds w+zw\HT\ds v)}|\\
&\le|\LR{u\HT D^{s}w,[D^{s},u^{2}]\HT\ds w}|+\frac{3}{2}|\LR{u^{2}\ds u,(\HT D^{s}w)^{2}}|
+|\LR{u\HT D^{s}w,D^{s}(zw\HT\ds v)}|.
}
We see that Lemma \ref{Leibniz2} completes the proof.
\end{proof}

By the presence of $\HT D^{s-1}$, the following lemma is clear:

\begin{lem}\label{lem3.45}
Let $s_{0}>7/2$ and $s\ge1$.
Let $u,v\in H^{\max\{s+1,s_{0}\}}(\T)$.
Then
\EQQS{
&|\LR{u\HT D^{s}\ds(u^{4}-v^{4}),\HT D^{s-1}w}|+|\LR{u\HT D^{s}w,D^{s}(u^{4}-v^{4})}|\\
&\lesssim I_{s_{0}}(u,v)^{3}\{\|w\|_{H^{s}}^{2}+(\|u\|_{H^{s}}^{2}+\|v\|_{H^{s}}^{2})\|w\|_{H^{s_{0}}}^{2}\},
}
where $w=u-v$.
\end{lem}

\begin{lem}\label{lem3.46}
Let $s_{0}>7/2$ and $s\ge1$.
Let $u,v\in H^{\max\{s+1,s_{0}\}}(\T)$.
Then
\EQQS{
&|\LR{(\HT\ds u)D^{s-1}w,D^{s-1}\ds(F_{3}(u)-F_{3}(v))}|+|\LR{(\HT\ds u)D^{s-1}w,D^{s-1}\ds(u^{4}-v^{4})}|\\
&\lesssim I_{s_{0}}(u,v)^{4}\{\|w\|_{H^{s}}^{2}+(\|u\|_{H^{s}}^{2}+\|v\|_{H^{s}}^{2})\|w\|_{H^{s_{0}}}^{2}\},
}
where $w=u-v$.
\end{lem}

\begin{proof}
This follows from Lemma \ref{Leibniz} because of the presence of $D^{s-1}$.
\end{proof}

\begin{lem}\label{lem4.47}
Let $s_{0}>7/2$ and $s\ge1$.
Let $u,v\in H^{\max\{s+2,s_{0}\}}(\T)$.
Then
\EQQS{
&|\LR{u^{2}D^{s-1}w,D^{s-1}\ds(u\ds^{2}u-v\ds^{2}v)}|\\
&\lesssim I_{s_{0}}(u,v)^{3}\{\|w\|_{H^{s}}^{2}+\|w\|_{H^{s_{0}-1}}^{2}\|v\|_{H^{s+1}}^{2}+(\|u\|_{H^{s}}^{2}+\|v\|_{H^{s}}^{2})\|w\|_{H^{s_{0}}}^{2}\},
}
where $w=u-v$.
\end{lem}

\begin{proof}
Note that
\EQQS{
&|\LR{u^{2}D^{s-1}w,D^{s-1}\ds(u\ds^{2}w+w\ds^{2}v)}|\\
&\le|\LR{u^{2}D^{s-1}w,D^{s-1}\ds(u\ds^{2}w)-uD^{s-1}\ds^{3}w-s\ds uD^{s-1}\ds^{2}w}|\\
&\quad+|\LR{\ds(u^{2}D^{s-1}w),D^{s-1}(w\ds^{2}v)}|+3|\LR{\ds(u^{2}\ds uD^{s-1}w),D^{s-1}\ds w}|\\
&\quad+\frac{3}{2}|\LR{u^{2}\ds u,(\HT D^{s}w)^{2}}|+s|\LR{\ds(u^{2}\ds uD^{s-1}w),D^{s-1}\ds w}|,
}
which completes the proof.
\end{proof}

\begin{lem}\label{lem4.48}
Let $s_{0}>7/2$ and $s\ge1$.
Let $u,v\in H^{\max\{s+2,s_{0}\}}(\T)$.
Then
\EQQS{
&|\LR{u^{2}D^{s-1}w,\HT D^{s-1}\ds(u\HT\ds^{2}u-v\HT\ds^{2}v)}|\\
&\lesssim I_{s_{0}}(u,v)^{3}\{\|w\|_{H^{s}}^{2}+\|w\|_{H^{s_{0}-1}}^{2}\|v\|_{H^{s+1}}^{2}+(\|u\|_{H^{s}}^{2}+\|v\|_{H^{s}}^{2})\|w\|_{H^{s_{0}}}^{2}\},
}
where $w=u-v$.
\end{lem}

\begin{proof}
This follows from a similar argument to Lemma \ref{lem4.47}.
\end{proof}

Cobmining Lemma \ref{lem4.47} and \ref{lem4.48}, we obtain the following:
\begin{lem}\label{lem3.49}
Let $s_{0}>7/2$ and $s\ge1$.
Let $u,v\in H^{\max\{s+2,s_{0}\}}(\T)$.
Then
\EQQS{
&\sum_{j=2}^{4}|\LR{u^{2}D^{s-1}w,D^{s-1}\ds(F_{j}(u)-F_{j}(v))}|\\
&\lesssim I_{s_{0}}(u,v)^{5}\{\|w\|_{H^{s}}^{2}+\|w\|_{H^{s_{0}-1}}^{2}\|v\|_{H^{s+1}}^{2}+(\|u\|_{H^{s}}^{2}+\|v\|_{H^{s}}^{2})\|w\|_{H^{s_{0}}}^{2}\},
}
where $w=u-v$.
\end{lem}

\begin{defn}
Let $s\ge0$ and $k\in\N$ satisfy $2(s+2)>k$.
We define
\EQQS{
p(k):=\frac{2(s+2)}{2(s+2)-k},\quad q(k):=\frac{2(s+2)}{k}.
}
Note that $p(k)>1$ and $1/p(k)+1/q(k)=1$.
\end{defn}

The following five lemmas are estimates for viscous terms $-\e_{1}\ds^{4}u+\e_{2}\ds^{4}v$ in $M_{s}^{(1)}(u,v)$.

\begin{lem}\label{lem3.50}
Let $s\ge1$, $s_{0}>7/2$ and $\e_{1}\in[0,1]$.
Then there exists $C=C(s_{0},s)>0$ such that
for any $u,v\in H^{\max\{s+2,s_{0}\}}(\T)$,
\EQQS{
&\left|\e_{1}\int_{\T}\ds^{4}u(\HT D^{s}w)\HT D^{s-1}wdx\right|\\
&\le\frac{\e_{1}^{p(4)}}{100}\|D^{s+2}w\|^{2}+C\|u\|_{H^{s_{0}}}^{q(4)}\|w\|^{2}
  +C\|u\|_{H^{s_{0}}}\|w\|_{H^{s}}^{2},
}
where $w=u-v$.
\end{lem}

\begin{proof}
We set
\EQQS{
\int_{\T}\ds^{4}u(\HT D^{s}w)\HT D^{s-1}wdx
&=-\int_{\T}\ds^{3}uD^{s+1}w\HT D^{s-1}wdx-\int_{\T}\ds^{3}u(\HT D^{s}w)D^{s}wdx\\
&=:A+B.
}
It is clear that $|B|\lesssim\|u\|_{H^{s_{0}}}\|w\|_{H^{s}}^{2}$
Interpolation and the Young inequality show that
\EQQS{
\e_{1}|A|
&\le \e_{1}C\|u\|_{H^{s_{0}}}\|D^{s+1}w\|\|D^{s-1}w\|\\
&\le \e_{1}C\|u\|_{H^{s_{0}}}\|w\|^{4/(s+2)}\|D^{s+2}w\|^{2-4/(s+2)}
\le\frac{\e_{1}^{p(4)}}{100}\|D^{s+2}w\|^{2}+C\|u\|_{H^{s_{0}}}^{q(4)}\|w\|^{2},
}
as desired.
\end{proof}

By a similar argument to the proof of Lemma \ref{lem3.50}, we can show the rest of estimates for viscous terms in $M_{s}^{(1)}(u,v)$.

\begin{lem}\label{lem3.51}
Let $s\ge1$, $s_{0}>7/2$ and $\e_{1}\in[0,1]$.
Then there exists $C=C(s_{0},s)>0$ such that
for any $u,v\in H^{\max\{s+4,s_{0}\}}(\T)$,
\EQQS{
\left|\e_{1}\int_{\T}u(\HT D^{s}\ds^{4}w)\HT D^{s-1}wdx\right|
\le\frac{\sum_{j=1}^{3}\e_{1}^{p(j)}}{100}\|D^{s+2}w\|^{2}+C\sum_{j=1}^{3}\|u\|_{H^{s_{0}}}^{q(j)}\|w\|^{2},
}
where $w=u-v$.
\end{lem}


\begin{lem}\label{lem3.52}
Let $s\ge1$, $s_{0}>7/2$ and $\e_{1}\in[0,1]$.
Then there exists $C=C(s_{0},s)>0$ such that
for any $u,v\in H^{\max\{s+3,s_{0}\}}(\T)$,
\EQQS{
\left|\e_{1}\int_{\T}u(\HT D^{s}w)D^{s}\ds^{3}wdx\right|
\le\frac{\sum_{j=1}^{2}\e_{1}^{p(j)}}{100}\|D^{s+2}w\|^{2}+C\sum_{j=1}^{2}\|u\|_{H^{s_{0}}}^{q(j)}\|w\|^{2},
}
where $w=u-v$.
\end{lem}



\begin{lem}\label{lem3.53}
Let $s\ge1$, $s_{0}>7/2$ and $\e_{1},\e_{2}\in[0,1]$.
Then there exists $C=C(s_{0},s)>0$ such that
for any $u,v\in H^{\max\{s+4,s_{0}\}}(\T)$,
\EQQS{
\left|(\e_{1}-\e_{2})\int_{\T}u(\HT D^{s}\ds^{4}v)\HT D^{s-1}wdx\right|
\le C\max\{\e_{1}^{2},\e_{2}^{2}\}\|v\|_{H^{s+4}}^{2}+C\|u\|_{H^{s_{0}}}^{2}\|w\|_{H^{s}}^{2},
}
where $w=u-v$.
\end{lem}

\begin{lem}\label{lem3.54}
Let $s\ge1$, $s_{0}>7/2$ and $\e_{1},\e_{2}\in[0,1]$.
Then there exists $C=C(s_{0},s)>0$ such that
for any $u,v\in H^{\max\{s+4,s_{0}\}}(\T)$,
\EQQS{
\left|(\e_{1}-\e_{2})\int_{\T}u(\HT D^{s}w)D^{s}\ds^{3}vdx\right|
\le C\max\{\e_{1}^{2},\e_{2}^{2}\}\|v\|_{H^{s+3}}^{2}+C\|u\|_{H^{s_{0}}}^{2}\|w\|_{H^{s}}^{2},
}
where $w=u-v$.
\end{lem}

The following three lemmas are estimates for viscous terms $-\e_{1}\ds^{4}u+\e_{2}\ds^{4}v$ in $M_{s}^{(2)}(u,v)$.
We omit the proofs of these lemmas since they are similar to that of Lemma \ref{lem3.50}.

\begin{lem}\label{lem3.55}
Let $s\ge1$, $s_{0}>7/2$ and $\e_{1}\in[0,1]$.
Then there exists $C=C(s_{0},s)>0$ such that
for any $u,v\in H^{\max\{s+4,s_{0}\}}(\T)$,
\EQQS{
\left|\e_{1}\int_{\T}(\HT\ds^{5}u)(D^{s-1}w)^{2}dx\right|
\le\frac{\e_{1}^{p(4)}}{100}\|D^{s+2}w\|^{2}+C\|u\|_{H^{s_{0}}}^{q(4)}\|w\|^{2}+C\|u\|_{H^{s_{0}}}\|w\|_{H^{s}}^{2},
}
where $w=u-v$.
\end{lem}


\begin{lem}\label{lem3.56}
Let $s\ge1$, $s_{0}>7/2$ and $\e_{1}\in[0,1]$.
Then there exists $C=C(s_{0},s)>0$ such that
for any $u,v\in H^{\max\{s+3,s_{0}\}}(\T)$,
\EQQS{
\left|\e_{1}\int_{\T}(\HT\ds u)D^{s-1}wD^{s-1}\ds^{4}wdx\right|
\le\frac{\sum_{j=2}^{3}\e_{1}^{p(j)}}{100}\|D^{s+2}w\|^{2}+C\sum_{j=2}^{3}\|u\|_{H^{s_{0}}}^{q(j)}\|w\|^{2},
}
where $w=u-v$.
\end{lem}

\begin{lem}\label{lem3.57}
Let $s\ge1$, $s_{0}>7/2$ and $\e_{1},\e_{2}\in[0,1]$.
Then there exists $C=C(s_{0},s)>0$ such that
for any $u,v\in H^{\max\{s+3,s_{0}\}}(\T)$,
\EQQS{
\left|(\e_{1}-\e_{2})\int_{\T}(\HT\ds u)D^{s-1}wD^{s-1}\ds^{4}vdx\right|
\le C\max\{\e_{1}^{2},\e_{2}^{2}\}\|v\|_{H^{s+3}}^{2}+C\|u\|_{H^{s_{0}}}^{2}\|w\|_{H^{s}}^{2},
}
where $w=u-v$.
\end{lem}

The following three lemmas are estimates for viscous terms $-\e_{1}\ds^{4}u+\e_{2}\ds^{4}v$ in $M_{s}^{(3)}(u,v)$.

\begin{lem}\label{lem3.58}
Let $s\ge1$, $s_{0}>7/2$ and $\e_{1}\in[0,1]$.
Then there exists $C=C(s_{0},s)>0$ such that
for any $u,v\in H^{\max\{s+3,s_{0}\}}(\T)$,
\EQQS{
\left|\e_{1}\int_{\T}u\ds^{4}u(D^{s-1}w)^{2}dx\right|\le C\|u\|_{H^{s_{0}}}^{2}\|w\|_{H^{s}}^{2},
}
where $w=u-v$.
\end{lem}

\begin{proof}
This is obvious thanks to the integration by parts.
\end{proof}

\begin{lem}\label{lem3.59}
Let $s\ge1$, $s_{0}>7/2$ and $\e_{1}\in[0,1]$.
Then there exists $C=C(s_{0},s)>0$ such that
for any $u,v\in H^{\max\{s+3,s_{0}\}}(\T)$,
\EQQS{
\left|\e_{1}\int_{\T}u^{2}D^{s-1}wD^{s-1}\ds^{4}wdx\right|
\le\frac{\sum_{j=2}^{3}\e_{1}^{p(j)}}{100}\|D^{s+2}w\|^{2}+C\sum_{j=2}^{3}\|u\|_{H^{s_{0}}}^{2q(j)}\|w\|^{2},
}
where $w=u-v$.
\end{lem}

\begin{proof}
First we set
\EQQS{
&\int_{\T}u^{2}D^{s-1}wD^{s-1}\ds^{4}wdx\\
&=-2\int_{\T}u\ds uD^{s-1}wD^{s-1}\ds^{3}wdx-\int_{\T}u^{2}D^{s-1}\ds wD^{s-1}\ds^{3}wdx=:A+B.
}
The same argument as before implies that
\EQQS{
&|A|\le C\|u\|_{H^{s_{0}}}^{2}\|w\|^{2/q(3)}\|D^{s+2}w\|^{2/p(3)},\\
&|B|\le C\|u\|_{H^{s_{0}}}^{2}\|w\|^{2/q(2)}\|D^{s+2}w\|^{2/p(2)},
}
which completes the proof.
\end{proof}

\begin{lem}\label{lem3.60}
Let $s\ge1$, $s_{0}>7/2$ and $\e_{1}\in[0,1]$.
Then there exists $C=C(s_{0},s)>0$ such that
for any $u,v\in H^{\max\{s+3,s_{0}\}}(\T)$,
\EQQS{
\left|(\e_{1}-\e_{2})\int_{\T}u^{2}D^{s-1}wD^{s-1}\ds^{4}vdx\right|
\le C\max\{\e_{1}^{2},\e_{2}^{2}\}\|v\|_{H^{s+3}}^{2}+C\|u\|_{H^{s_{0}}}^{4}\|w\|_{H^{s}}^{2},
}
where $w=u-v$.
\end{lem}

\begin{proof}
This follows from the H\"older inequality.
\end{proof}

Finally, we are ready to show the main inequality in this paper.

\begin{proof}[Proof of Proposition \ref{ene.est.pr}]
Let $s'\in[1,s]$.
Put $w:=u_{1}-u_{2}$.
Note that $w$ satisfies
\EQS{\label{eq3.7}
\dt w=\ds(K(u_{1})-K(u_{2}))-\e_{1}\ds^{4}w+(\e_{1}-\e_{2})\ds^{4}u_{2}
}
on $[0,\min\{T_{\e_{1}},T_{\e_{2}}\})$.
By Lemma \ref{lem2.3new}, \ref{lem2.4new}, \ref{lem2.5new}, \ref{lem2.6new}, \ref{lem2.8new}, \ref{lem2.9new}, \ref{lem2.10new} and \ref{lem.fo}, we have
\EQS{
\begin{aligned}\label{eq3.1}
&\left|\frac{1}{2}\frac{d}{dt}\|D^{s'}w\|^{2}+\la_{1}(s')\LR{\ds u_{1},(D^{s'}\ds w)^{2}}
 +\la_{2}(s')\LR{(\HT\ds^{2}u_{1})\HT D^{s'}\ds w,D^{s'}w}\right.\\
 &\quad+\la_{3}(s')\LR{u_{1}\ds u_{1}\HT D^{s'}\ds w,D^{s'}w}+2\e_{1}\|D^{s'+2}w\|^{2}|\\
&=|\LR{D^{s'}\{\ds(K(u_{1})-K(u_{2}))-\e_{1}\ds^{4}w+(\e_{1}-\e_{2})\ds^{4}u_{2}\},D^{s'}w}\\
&\quad+\la_{1}(s')\LR{\ds u_{1},(D^{s'}\ds w)^{2}}
 +\la_{2}(s')\LR{(\HT\ds^{2}u_{1})\HT D^{s'}\ds w,D^{s'}w}+2\e_{1}\|D^{s'+2}w\|^{2}|\\
&\le CI_{s_{0}}(u_{1},u_{2})^{3}\{\|w\|_{H^{s'}}^{2}+\|w\|_{H^{s_{0}-3}}^{2}\|u_{2}\|_{H^{s'+3}}^{2}+\|w\|_{H^{s_{0}}}^{2}(\|u_{1}\|_{H^{s'}}^{2}+\|u_{2}\|_{H^{s'}}^{2})\}\\
&\quad+\max\{\e_{1}^{2},\e_{2}^{2}\}\|u_{2}\|_{H^{s'+4}}^{2}.
\end{aligned}
}
By Lemma \ref{m.e.1.2new}, \ref{lem4.23}, \ref{lem4.25}, \ref{lem4.26}, \ref{lem4.28}, \ref{lem4.29}, \ref{lem4.30}, \ref{lem4.31}, \ref{lem4.32}, \ref{lem4.39}, \ref{lem4.41}, \ref{lem4.42}, \ref{lem4.43}, \ref{lem4.44}, \ref{lem4.45}, \ref{lem3.45}, \ref{lem3.50}, \ref{lem3.51}, \ref{lem3.52}, \ref{lem3.53} and \ref{lem3.54}, we also have
\EQS{
\begin{aligned}\label{eq3.2}
&\left|\frac{d}{dt}M_{s'}^{(1)}(u_{1},u_{2})-\la_{1}(s')\LR{\ds u_{1},(D^{s'}\ds w)^{2}}
 +\frac{\la_{1}(s')\la_{4}(s')}{4}\LR{(\HT\ds^{2}u_{1})\HT D^{s'}\ds w,D^{s'}w}\right|\\
&=\left|\frac{\la_{1}(s')}{4}(\LR{\dt u\HT D^{s'}w,\HT D^{s'-1}w}+\LR{u\HT D^{s'}\dt w,\HT D^{s'-1}w}+\LR{u\HT D^{s'}w,\HT D^{s'-1}\dt w})\right.\\
&\quad\left.-\la_{1}(s')\LR{\ds u_{1},(D^{s'}\ds w)^{2}}
 +\frac{\la_{1}(s')\la_{4}(s')}{4}\LR{(\HT\ds^{2}u_{1})\HT D^{s'}\ds w,D^{s'}w}\right|\\
&\le CI_{s_{0}}(u_{1},u_{2})^{2(s'+2)}\{\|w\|_{H^{s'}}^{2}+\|w\|_{H^{s_{0}-3}}^{2}\|u_{2}\|_{H^{s'+3}}^{2}+\|w\|_{H^{s_{0}}}^{2}(\|u_{1}\|_{H^{s'}}^{2}+\|u_{2}\|_{H^{s'}}^{2})\}\\
&\quad+\frac{\e_{1}}{10}\|D^{s'+2}w\|^{2}+\max\{\e_{1}^{2},\e_{2}^{2}\}\|u_{2}\|_{H^{s'+4}}^{2}.
\end{aligned}
}
Similarly, by Lemma \ref{lem2.14new}, \ref{lem4.34}, \ref{lem4.35}, \ref{lem4.37}, \ref{lem3.46}, \ref{lem3.55}, \ref{lem3.56}, \ref{lem3.57} and \ref{lem3.58}, we have
\EQS{
\begin{aligned}\label{eq3.3}
&\left|\frac{d}{dt}M_{s'}^{(2)}(u_{1},u_{2})-\la_{2}(s')\LR{(\HT\ds^{2}u_{1})\HT D^{s'}\ds w,D^{s'}w}\right|\\
&\le CI_{s_{0}}(u_{1},u_{2})^{2(s'+2)}\{\|w\|_{H^{s'}}^{2}+\|w\|_{H^{s_{0}-3}}^{2}\|u_{2}\|_{H^{s'+3}}^{2}\\
&\quad+\|w\|_{H^{s_{0}}}^{2}(\|u_{1}\|_{H^{s'}}^{2}+\|u_{2}\|_{H^{s'}}^{2})\}+\frac{\e_{1}}{10}\|D^{s'+2}w\|^{2}+\max\{\e_{1}^{2},\e_{2}^{2}\}\|u_{2}\|_{H^{s'+4}}^{2}.
\end{aligned}
}
Moreover, by Lemma \ref{lem2.21new}, \ref{lem3.49}, \ref{lem3.59} and \ref{lem3.60}, we obtain
\EQS{
\begin{aligned}\label{eq3.4}
&\left|\frac{d}{dt}M_{s'}^{(3)}(u_{1},u_{2})-\frac{\la_{1}(s')\la_{4}(s')+4\la_{3}(s')}{4}\LR{u\ds u\HT D^{s'}\ds w,D^{s'}w}\right|\\
&\le CI_{s_{0}}(u_{1},u_{2})^{2(s'+2)}\{\|w\|_{H^{s'}}^{2}+\|w\|_{H^{s_{0}-3}}^{2}\|u_{2}\|_{H^{s'+3}}^{2}\\
&\quad+\|w\|_{H^{s_{0}}}^{2}(\|u_{1}\|_{H^{s'}}^{2}+\|u_{2}\|_{H^{s'}}^{2})\}+\frac{\e_{1}}{10}\|D^{s'+2}w\|^{2}+\max\{\e_{1}^{2},\e_{2}^{2}\}\|u_{2}\|_{H^{s'+4}}^{2}.
\end{aligned}
}
It is easy to see that
\EQS{
\begin{aligned}\label{eq3.5}
\frac{d}{dt}\{\|w\|^{2}(1+C\|u_{1}\|^{2}+C\|u_{1}\|^{4s'})\}\le CI_{s_{0}}(u_{1},u_{2})^{4s'+3}\|w\|_{H^{s'}}^{2}.
\end{aligned}
}
Therefore, collecting \eqref{eq3.1}, \eqref{eq3.2}, \eqref{eq3.3}, \eqref{eq3.4} and \eqref{eq3.5}, we obtain \eqref{eq2.4}.
\end{proof}


\section{The energy estimate in $L^{2}$}
In this section, we prove Proposition \ref{ene.est.0}, which is the only thing left to prove.
We introduce some estimates for the operator $J$.

\begin{lem}\label{freq.est.2}
Let $k\in\N\cup\{0\}$.
There exists $C=C(k)>0$ such that
\[\|\HT J\ds^{k+1}f-\ds^{k}f\|\le C\|f\|\]
for any $f\in L^{2}(\T)$.
\end{lem}

\begin{proof}
It suffices to show that there exists $C=C(k)>0$ such that
\[\left|-i\sgn(\xi)\frac{\psi(\xi)}{|\xi|}(i\xi)^{k+1}-(i\xi)^{k}\right|\le C\]
for any $\xi\in\Z$.
But this is clear since the left hand side is equal to $|\xi|^{k}|\psi(\xi)-1|$ and $\supp(1-\psi)\subset\{|\xi|\le2\}$.
\end{proof}

\begin{lem}\label{reduction3}
Let $s_{0}>5/2$.
Let $u\in H^{s_{0}}(\T)$ and $v\in L^{2}(\T)$.
Then
\EQQS{
\|J\ds(u\ds^{2}v)+u\HT\ds^{2}v\|\lesssim\|u\|_{H^{s_{0}}}\|v\|.
}
\end{lem}

\begin{proof}
Note that
\EQQS{
&J\ds(u\ds^{2}v)+u\HT\ds^{2}v\\
&=(J\ds^{3}+\HT\ds^{2})(uv)-2(J\ds^{2}+\HT\ds)(\ds uv)+J\ds(\ds^{2}uv)-[\HT,u]\ds^{2}v+\HT(\ds^{2}uv).
}
Lemma \ref{freq.est.2} shows the desired inequality.
\end{proof}

As a corollary, we have the following.

\begin{cor}\label{reduction3.H}
Let $s_{0}>5/2$.
Let $u\in H^{s_{0}}(\T)$ and $v\in L^{2}(\T)$.
Then
\EQQS{
\|\HT J\ds(u\HT\ds^{2}v)-u\HT\ds^{2}v\|\lesssim\|u\|_{H^{s_{0}}}\|v\|.
}
\end{cor}

\begin{lem}\label{reduction4}
Let $s_{0}>5/2$.
Let $u\in H^{s_{0}}(\T)$ and $v\in L^{2}(\T)$.
Then
\EQQS{
\|J\ds(\ds u\ds v)+\ds u\HT\ds v\|\lesssim\|u\|_{H^{s_{0}}}\|v\|.
}
\end{lem}

\begin{proof}
This follows from the follwing equality
\EQQS{
&J\ds(\ds u\ds v)+\ds u\HT\ds v\\
&=(J\ds^{2}+\HT\ds)(\ds uv)-J\ds(\ds^{2}uv)-\HT(\ds^{2}uv)-[\HT,\ds u]\ds v
}
and Lemma \ref{freq.est.2}.
\end{proof}

As a corollary, we have the following.

\begin{cor}\label{reduction4.H}
Let $s_{0}>5/2$.
Let $u\in H^{s_{0}}(\T)$ and $v\in L^{2}(\T)$.
Then
\EQQS{
\|\HT J\ds((\HT\ds u)\HT\ds v)-(\HT\ds u)\HT\ds v\|\lesssim\|u\|_{H^{s_{0}}}\|v\|.
}
\end{cor}

\begin{lem}\label{comm.est.4.1}
Let $s_{0}>1/2$ and $\La=D^{2}$ or $D\ds$.
There exists $C(s_{0})>0$ such that for any $f\in H^{s_{0}+1}(\T)$ and $g\in L^{2}(\T)$,
\[\|[\LR{D}^{-1}\La,f]g\|\le C\|f\|_{H^{s_{0}+1}}\|g\|.\]
\end{lem}

\begin{proof}
See $(ii)$ of Lemma 2.4 in \cite{Tanaka}.
\end{proof}

We estimate the time derivative of $M^{(j)}(u,v)$ for $j=1,2,3$.

\begin{lem}\label{m.e.1.2new0}
Let $s_{0}>7/2$.
Let $u,w\in H^{s_{0}+1}(\T)$.
Then
\EQQS{
|\LR{(\HT\ds^{4}u)\HT w,\HT Jw}-\LR{u\ds^{4}w,\HT Jw}-\LR{u\HT w,J\ds^{4}w}-4\LR{\ds u,(\ds w)^{2}}|
\lesssim\|u\|_{H^{s_{0}}}\|w\|^{2}.
}
\end{lem}

\begin{proof}
We use Lemma \ref{fourth} with $f=u$, $g=\HT w$ and $h=\HT Jw$.
It is clear that \eqref{eq3.6} with $s=0$ holds.
Set
\EQQS{
-4\LR{\ds f\HT\ds g,\ds^{2}h}+2\LR{\ds^{2}f\HT g,\ds^{2}h}
=6\LR{\ds^{2}f\HT g,\ds^{2}h}+4\LR{\ds f\HT g,\ds^{3}h}=:A+B.
}
For $A$, Lemma \ref{freq.est.2} shows that
\EQQS{
|A|
&\le6|\LR{w\ds^{2}u,(\HT J\ds^{2}-\ds)w}|+3|\LR{\ds^{3}u,w^{2}}|+6|\hat{w}(0)\LR{\ds^{3}u,\HT J\ds w}|\\
&\lesssim\|u\|_{H^{s_{0}}}\|w\|^{2}.
}
Similarly, Lemma \ref{lem4.5} and \ref{freq.est.2} show that
\EQQS{
&|B-4\LR{\ds u,(\ds w)^{2}}|\\
&\le4|\LR{w\ds u,(\HT J\ds^{3}-\ds^{2})w}|+4|\LR{\ds uw,\ds^{2}w}+\LR{\ds u,(\ds w)^{2}}|
\lesssim\|u\|_{H^{s_{0}}}\|w\|^{2},
}
which completes the proof.
\end{proof}



\begin{lem}\label{lem2.14new0}
Let $s_{0}>7/2$.
Let $u,w\in H^{s_{0}+1}(\T)$.
Then
\EQQS{
|\LR{\ds^{5}u,(Jw)^{2}}-2\LR{(\HT\ds u)Jw,\HT J\ds^{4}w}+4\LR{(\HT\ds^{2}u)\HT\ds w,w}|
\lesssim\|u\|_{H^{s_{0}}}\|w\|^{2}.
}
\end{lem}

\begin{proof}
The integration by parts shows that
\EQQS{
&|\LR{\ds^{5}u,(Jw)^{2}}-2\LR{(\HT\ds u)Jw,\HT J\ds^{4}w}+4\LR{(\HT\ds^{2}u)\HT\ds w,w}|\\
&=|\LR{\ds^{4}uJw,J\ds w}-\LR{(\HT u)J\ds w,\HT J\ds^{4}w}-\LR{(\HT u)Jw,\HT J\ds^{5}w}\\
&\quad-2\LR{(\HT\ds^{2}u)\HT\ds w,w}|,
}
which allows us to use Lemma \ref{fourth} with $f=\HT u$, $g=J\ds w$ and $h=Jw$.
Then \eqref{eq3.6} with $s=0$ holds.
Lemma \ref{reduction} shows that
\EQQS{
|\LR{\ds f\HT\ds g,\ds^{2}h}|=|\LR{\ds((\HT\ds u)\HT J\ds^{2}w),J\ds w}|\lesssim\|u\|_{H^{s_{0}}}\|w\|^{2}.
}
Finally, we have
\EQQS{
&2\LR{\ds^{2}f\HT g,\ds^{2}h}\\
&=2\LR{(\HT\ds^{2}u)\HT J\ds w,J\ds^{2}w}\\
&=2\LR{(\HT\ds^{2}u)\HT J\ds w,(J\ds^{2}+\HT\ds)w}+2\LR{(\HT\ds^{3}u)(\HT J\ds-1)w,\HT w}\\
&\quad+2\LR{(\HT\ds^{2}u)(\HT J\ds^{2}-\ds)w,\HT w}-2\LR{(\HT\ds^{2}u)\HT\ds w,w},
}
which completes the proof.
\end{proof}



\begin{lem}\label{lem2.21new0}
Let $s_{0}>7/2$.
Let $u,w\in H^{s_{0}+1}(\T)$.
Then
\EQQS{
|\LR{u\HT\ds^{4}u,(Jw)^{2}}+\LR{u^{2}Jw,\HT J\ds^{4}w}-4\LR{u\ds u\HT\ds w,w}|
\lesssim\|u\|_{H^{s_{0}}}^{2}\|w\|^{2}.
}
\end{lem}

\begin{proof}
Adding and subtraction a term, we have
\EQQS{
&2|\LR{u\HT\ds^{4}u,(Jw)^{2}}+\LR{u^{2}Jw,\HT J\ds^{4}w}-4\LR{u\ds u\HT\ds w,w}|\\
&\le|\LR{\HT\ds^{4}(u^{2}),(Jw)^{2}}+2\LR{u^{2}Jw,\HT J\ds^{4}w}-8\LR{u\ds u\HT\ds w,w}|\\
&\quad+|\LR{2u\HT\ds^{4}u-\HT\ds^{4}(u^{2}),(Jw)^{2}}|.
}
Lemma \ref{H^-1} and \ref{comm.est.H} show that the second term in the right hand side can be estimated by $\lesssim\|u\|_{H^{s_{0}}}^{2}\|w\|^{2}$.
We use Lemma \ref{fourth} with $f=u^{2}$ and $g=h=Jw$.
Note that \eqref{eq3.6} with $s=0$ holds.
It is easy to see that
\EQQS{
|\LR{\ds^{2}f\HT g,\ds^{2}h}|=|\LR{\ds(\ds^{2}(u^{2})\HT Jw),J\ds w}|\lesssim\|u\|_{H^{s_{0}}}^{2}\|w\|^{2}.
}
Finally, Lemma \ref{freq.est.2} shows that
\EQQS{
&-4\LR{\ds f\HT\ds g,\ds^{2}h}\\
&=-8\LR{u\ds u\HT J\ds w,(J\ds^{2}+\HT\ds)w}-8\LR{\ds(u\ds u)(\HT J\ds-1)w,\HT w}\\
&\quad-8\LR{u\ds u(\HT J\ds^{2}-\ds)w,\HT w}+8\LR{u\ds u\HT\ds w,w},
}
which completes the proof.
\end{proof}


\begin{lem}\label{lem4.23.0}
Let $s_{0}>7/2$.
Let $u,v\in H^{s_{0}}(\T)$.
Then
\EQQS{
|\LR{u\HT\ds(u\ds^{2}u-v\ds^{2}v),\HT Jw}-3\LR{u\ds u\HT\ds w,w}|
\lesssim I_{s_{0}}(u,v)^{2}\|w\|^{2},
}
where $w=u-v$.
\end{lem}

\begin{proof}
First we set
\EQQS{
A:=\LR{u\HT\ds(u\ds^{2}w),\HT Jw},\quad B:=\LR{u\HT\ds(w\ds^{2}v),\HT Jw}.
}
It is clear that $|B|\lesssim\|u\|_{H^{s_{0}}}\|v\|_{H^{s_{0}}}\|w\|^{2}$.
Note that
\[\ds(u\ds^{2}w)=\ds^{3}(uw)-2\ds^{2}(\ds uw)+\ds(\ds^{2}uw).\]
Then we set
\EQQS{
A
&=\LR{u\HT\ds^{3}(uw),\HT Jw}-2\LR{u\HT\ds^{2}(\ds uw),\HT Jw}+\LR{u\HT\ds(\ds^{2}uw),\HT Jw}\\
&=:A_{1}+A_{2}+A_{3}.
}
It is clear that $|A_{3}|\lesssim\|u\|_{H^{s_{0}}}^{2}\|w\|^{2}$.
For $A_{1}$, we have
\EQQS{
A_{1}
&=\LR{uw,\HT\ds^{3}(u\HT Jw)}\\
&=\LR{uw,\HT(\ds^{3}u\HT Jw)}+3\LR{uw,\HT(\ds^{2}u\HT J\ds w)}+3\LR{uw,\HT(\ds u\HT J\ds^{2}w)}\\
&\quad+\LR{uw,\HT(u\HT J\ds^{3}w)}=:A_{11}+\cdots+A_{14}.
}
It is clear that $|A_{11}|+|A_{12}|\lesssim\|u\|_{H^{s_{0}}}^{2}\|w\|^{2}$.
For $A_{13}$, we have
\EQQS{
A_{13}
&=3\LR{uw,[\HT,\ds u]\HT J\ds^{2}w}-3\LR{u\ds uw,J\ds^{2}w}\\
&=3\LR{uw,[\HT,\ds u]\HT J\ds^{2}w}-3\LR{u\ds uw,J\ds^{2}w+\HT\ds w}+3\LR{u\ds u\HT\ds w,w}.
}
Similarly, we have
\EQQS{
A_{14}
&=\LR{uw,[\HT,u]\HT J\ds^{3}w}-\LR{u^{2}w,J\ds^{3}+\HT\ds^{2}w}\\
&\quad+\LR{\ds(u^{2}\HT\ds w),w}-2\LR{u\ds u\HT\ds w,w}.
}
Finally, we have
\EQQS{
A_{2}
&=2\LR{\ds uw,\HT(\ds^{2}u\HT Jw)}+4\LR{\ds uw,\HT(\ds u\HT J\ds w)}+2\LR{\ds uw,\HT(u\HT J\ds^{2}w)}\\
&=:A_{21}+A_{22}+A_{23}.
}
Obviously, $|A_{21}|+|A_{22}|\lesssim\|u\|_{H^{s_{0}}}^{2}\|w\|^{2}$.
Observe that
\EQQS{
A_{23}=2\LR{\ds uw,[\HT,u]\HT J\ds^{2}w}-2\LR{u\ds u(J\ds^{2}w+\HT\ds w),w}+2\LR{u\ds u\HT\ds w,w}.
}
Therefore, we have
\EQQS{
&|A+B-3\LR{u\ds u\HT\ds w,w}|\\
&\le|A_{1}-\LR{u\ds u\HT\ds w,w}|+|A_{2}-2\LR{u\ds u\HT\ds w,w}|+|A_{3}|+|B|
\lesssim I_{s_{0}}(u,v)^{2}\|w\|^{2},
}
which completes the proof.
\end{proof}

\begin{lem}\label{lem4.25.0}
Let $s_{0}>7/2$.
Let $u,v\in H^{s_{0}+1}(\T)$.
Then
\EQQS{
|\LR{u\HT w,\HT J\ds(u\ds^{2}u-v\ds^{2}v)}-2\LR{u\ds u\HT\ds w,w}|
\lesssim I_{s_{0}}(u,v)^{2}\|w\|^{2},
}
where $w=u-v$.
\end{lem}

\begin{proof}
First we set $A:=\LR{u\HT w,\HT J\ds(u\ds^{2}w)}$ and $\quad B:=\LR{u\HT w,\HT J\ds(w\ds^{2}v)}$.
It is easy to see that $|B|\lesssim I(u,v)^{2}\|w\|^{2}$.
We have
\EQQS{
A
&=\LR{u\HT w,\HT\{J\ds(u\ds^{2}w)+u\HT\ds^{2}w\}}-\LR{u\HT w,[\HT,u]\ds^{2}w}\\
&\quad+\LR{\ds^{2}(u^{2})\HT w,w}+\LR{\ds(u^{2}\HT\ds w),w}+2\LR{u\ds u\HT\ds w,w}.
}
Lemma \ref{comm.est.H}, \ref{reduction} and \ref{reduction3} show that
\EQQS{
|A-2\LR{u\ds u\HT\ds w,w}|\lesssim I_{s_{0}}(u,v)^{2}\|w\|^{2},
}
which completes the proof.
\end{proof}

We modify Lemma \ref{lem4.26} in $L^{2}(\T)$.

\begin{lem}\label{lem4.26.0}
Let $s_{0}>7/2$.
Let $u,v\in H^{s_{0}}(\T)$.
Then
\EQQS{
|\LR{u\HT\ds((\ds u)^{2}-(\ds v)^{2}),\HT Jw}+2\LR{u\ds u\HT\ds w,w}|
\lesssim I_{s_{0}}(u,v)^{2}\|w\|^{2},
}
where $w=u-v$.
\end{lem}

\begin{proof}
Set $z=u+v$.
First we set 
\EQQS{
\LR{u\HT\ds(\ds z\ds w),\HT Jw}
&=\LR{u\HT\ds^{2}(\ds zw),\HT Jw}-\LR{u\HT\ds(\ds^{2}zw),\HT Jw}\\
&=:A+B.
}
It is clear that $|B|\lesssim I(u,v)^{2}\|w\|^{2}$.
Moreover, we set
\EQQS{
A
&=-\LR{\ds zw,\HT\ds^{2}(u\HT Jw)}\\
&=-\LR{\ds zw,\HT(\ds^{2}u\HT Jw)}-2\LR{\ds zw,\HT(\ds u\HT J\ds w)}-\LR{\ds zw,\HT(u\HT J\ds^{2}w)}\\
&=:A_{1}+A_{2}+A_{3}.
}
It is also clear that $|A_{1}|+|A_{2}|\lesssim I(u,v)^{2}\|w\|^{2}$.
For $A_{3}$, we have
\EQQS{
A_{3}
&=-\LR{\ds zw,[\HT,u]\HT J\ds^{2}w}+\LR{\ds zw,uJ\ds^{2}w+\HT\ds w}-2\LR{u\ds u\HT\ds w,w}\\
&\quad-\frac{1}{2}\LR{\ds(u\HT\ds w),w^{2}},
}
from which follows that
\EQQS{
&|A+B+2\LR{u\ds u\HT\ds w,w}|\\
&\le|A_{1}|+|A_{2}|+|A_{3}+2\LR{u\ds u\HT\ds w,w}|+|B|
\lesssim I_{s_{0}}(u,v)^{2}\|w\|^{2},
}
which completes the proof.
\end{proof}

\begin{lem}\label{lem4.28.0}
Let $s_{0}>7/2$.
Let $u,v\in H^{s_{0}}(\T)$.
Then
\EQQS{
|\LR{u\HT w,\HT J\ds((\ds u)^{2}-(\ds v)^{2})}+2\LR{u\ds u\HT\ds w,w}|
\lesssim I_{s_{0}}(u,v)^{2}\|w\|^{2},
}
where $w=u-v$.
\end{lem}

\begin{proof}
Set $z=u+v$.
Note that
\EQQS{
&\LR{u\HT w,\HT J\ds(\ds z\ds w)}\\
&=\LR{u\HT w,\HT\{J\ds (\ds z\ds w)+\ds z\HT\ds w\}}-\LR{u\HT w,[\HT,\ds z]\HT\ds w}\\
&\quad-\LR{\ds(u\ds z)\HT w,w}-\LR{\ds(u\HT\ds w),w^{2}}/2-2\LR{u\ds u\HT\ds w,w}.
}
Then, Lemma \ref{comm.est.H} together with Lemma \ref{reduction4} completes the proof.
\end{proof}

\begin{lem}\label{lem4.29.0}
Let $s_{0}>7/2$.
Let $u,v\in H^{s_{0}}(\T)$.
Then
\EQQS{
|\LR{u\HT\ds((\HT\ds u)^{2}-(\HT\ds v)^{2}),\HT Jw}|
\lesssim I_{s_{0}}(u,v)^{2}\|w\|^{2},
}
where $w=u-v$.
\end{lem}

\begin{proof}
Set $z=u+v$ and set
\EQQS{
&\LR{u\HT\ds((\HT\ds z)\HT\ds w),\HT Jw}\\
&=\LR{u\HT\ds^{2}((\HT\ds z)\HT w),\HT Jw}-\LR{u\HT\ds((\HT\ds^{2}z)\HT w),\HT Jw}
=:A+B.
}
It is clear that $|B|\lesssim I_{s_{0}}(u,v)^{2}\|w\|^{2}$.
Observe that
\EQQS{
A
&=-\LR{(\HT\ds z)\HT w,\HT\ds^{2}(u\HT Jw)}\\
&=-\LR{(\HT\ds z)\HT w,\HT(\ds^{2}u\HT Jw)}-2\LR{(\HT\ds z)\HT w,\HT(\ds u\HT J\ds w)}\\
&\quad-\LR{(\HT\ds z)\HT w,\HT(u\HT J\ds^{2}w)}=:A_{1}+A_{2}+A_{3}.
}
Again, $|A_{1}|+|A_{2}|\lesssim I_{s_{0}}(u,v)^{2}\|w\|^{2}$.
And we note that
\EQQS{
A_{3}
&=-\LR{(\HT\ds z)\HT w,[\HT,u]\HT J\ds^{2}w}+\LR{(\HT\ds z)\HT w,uJ\ds^{2}w}\\
&=-\LR{(\HT\ds z)\HT w,[\HT,u]\HT J\ds^{2}w}+\LR{u(\HT\ds z)\HT w,J\ds^{2}w+\HT\ds w}\\
&\quad+\frac{1}{2}\LR{\ds(u\HT\ds z),(\HT w)^{2}},
}
which completes the proof.
\end{proof}

\begin{lem}\label{lem4.30.0}
Let $s_{0}>7/2$.
Let $u,v\in H^{s_{0}}(\T)$.
Then
\EQQS{
|\LR{u\HT w,\HT J\ds((\HT\ds u)^{2}-(\HT\ds v)^{2})}|
\lesssim I_{s_{0}}(u,v)^{2}\|w\|^{2},
}
where $w=u-v$.
\end{lem}

\begin{proof}
Set $z=u+v$.
Note that
\EQQS{
&\LR{u\HT w,\HT J\ds((\HT\ds z)\HT\ds w)}\\
&=\LR{u\HT w,\HT J\ds((\HT\ds z)\HT\ds w)-(\HT\ds z)\HT\ds w}-\frac{1}{2}\LR{\ds(u\HT\ds z),(\HT w)^{2}}.
}
Corollary \ref{reduction4.H} completes the proof.
\end{proof}

\begin{lem}\label{lem4.31.0}
Let $s_{0}>7/2$.
Let $u,v\in H^{s_{0}}(\T)$.
Then
\EQQS{
|\LR{u\ds(u\HT\ds^{2}u-v\HT\ds^{2}v),\HT Jw}-3\LR{u\ds u\HT\ds w,w}|
\lesssim I_{s_{0}}(u,v)^{2}\|w\|^{2},
}
where $w=u-v$.
\end{lem}

\begin{proof}
First we set
\EQQS{
A:=\LR{u\ds(u\HT\ds^{2}w),\HT Jw},\quad B:=\LR{u\ds(w\HT\ds^{2}v),\HT Jw}.
}
It is clear that $|B|\lesssim I_{s_{0}}(u,v)^{2}\|w\|^{2}$.
We also set
\EQQS{
A
&=\LR{u\ds^{3}(u\HT w),\HT Jw}-2\LR{u\ds^{2}(\ds u\HT w),\HT Jw}+\LR{u\ds(\ds^{2}u\HT w),\HT Jw}\\
&=:A_{1}+A_{2}+A_{3}.
}
Again, it is clear that $|A_{3}|\lesssim I_{s_{0}}(u,v)^{2}\|w\|^{2}$.
Note that
\EQQS{
A_{1}
&=-\LR{u\HT w,\ds^{3}(u\HT Jw)}\\
&=-\LR{u\HT w,\ds^{3}u\HT Jw}-3\LR{u\HT w,\ds^{2}u\HT J\ds w}-3\LR{u\HT w,\ds u\HT J\ds^{2}w}\\
&\quad-\LR{u^{2}\HT w,\HT J\ds^{3}w}=:A_{11}+A_{12}+A_{13}+A_{14}.
}
It is easy to see that $|A_{11}|+|A_{12}|\lesssim I_{s_{0}}(u,v)^{2}\|w\|^{2}$.
We have
\EQQS{
A_{13}
=-3\LR{u\ds u\HT w,(\HT J\ds^{2}-\ds)w}+3\LR{\ds(u\ds u)\HT w,w}+3\LR{u\ds u\HT\ds w,w}
}
and
\EQQS{
A_{14}
&=-\LR{u^{2}\HT w,(\HT J\ds^{3}-\ds^{2})w}-\LR{\ds^{2}(u^{2})\HT w,w}-\LR{\ds(u^{2}\HT\ds w),w}\\
&\quad-2\LR{u\ds u\HT\ds w,w}.
}
Similarly, we have
\EQQS{
A_{2}
&=-2\LR{\ds u\HT w,\ds^{2}u\HT Jw+2\ds u\HT J\ds w}-2\LR{u\ds u\HT w,\HT J\ds^{2}w}\\
&=-2\LR{\ds u\HT w,\ds^{2}u\HT Jw+2\ds u\HT J\ds w}-2\LR{u\ds u\HT w,(\HT J\ds^{2}-\ds)w}\\
&\quad+2\LR{\ds(u\ds u)\HT w,w}+2\LR{u\ds u\HT\ds w,w}.
}
Therefore, we see that
\EQQS{
&|A+B-3\LR{u\ds u\HT\ds w,w}|\\
&\le|A_{11}|+|A_{12}|+|A_{13}-3\LR{u\ds u\HT\ds w,w}|+|A_{14}+2\LR{u\ds u\HT\ds w,w}|\\
&\quad+|A_{2}-2\LR{u\ds u\HT\ds w,w}|+|A_{3}|+|B|\lesssim I_{s_{0}}(u,v)^{2}\|w\|^{2},
}
which completes the proof.
\end{proof}

\begin{lem}\label{lem4.32.0}
Let $s_{0}>7/2$.
Let $u,v\in H^{s_{0}}(\T)$.
Then
\EQQS{
|\LR{u\HT w,J\ds(u\HT\ds^{2}u-v\HT\ds^{2}v)}-2\LR{u\ds u\HT\ds w,w}|
\lesssim I_{s_{0}}(u,v)^{2}\|w\|^{2},
}
where $w=u-v$.
\end{lem}

\begin{proof}
First we set $A=\LR{u\HT w,J\ds(u\HT\ds^{2}w)}$ and $\quad B:=\LR{u\HT w,J\ds(w\ds^{2}v)}$.
It is clear that $|B|\lesssim I(u,v)^{2}\|w\|^{2}$.
On the other hand, we have
\EQQS{
A
&=\LR{u\HT w,J\ds(u\HT\ds^{2}w)-u\ds^{2}w}+\LR{\ds^{2}(u^{2})\HT w,w}+\LR{\ds(u^{2}\HT\ds w),w}\\
&\quad+2\LR{u\ds u\HT\ds w,w}.
}
Then, Lemma \ref{reduction3} completes the proof.
\end{proof}

\begin{lem}\label{lem4.34.0}
Let $s_{0}>7/2$.
Let $u,v\in H^{s_{0}}(\T)$.
Then
\EQQS{
|\LR{(\HT\ds u)Jw,J\ds(u\ds^{2}u-v\ds^{2}v)}|
\lesssim I_{s_{0}}(u,v)^{2}\|w\|^{2},
}
where $w=u-v$.
\end{lem}

\begin{proof}
First we set $A=\LR{(\HT\ds u)Jw,J\ds(u\ds^{2}w)}$ and $B=\LR{(\HT\ds u)Jw,J\ds(w\ds^{2}v)}.$
It is easy to see that $|B|\lesssim I_{s_{0}}(u,v)^{2}\|w\|^{2}$.
We also set $A'=\LR{(\HT\ds u)Jw,u\HT\ds^{2}w}$.
Lemma \ref{reduction3} shows that $|A+A'|\lesssim I_{s_{0}}(u,v)^{2}\|w\|^{2}$.
So we consider $A'$.
Note that
\EQQS{
A'
&=\LR{\ds^{2}(u\HT\ds u)Jw,\HT w}+2\LR{\ds(u\HT\ds u)J\ds w,\HT w}+\LR{u(\HT\ds u)J\ds^{2}w,\HT w}\\
&=:A_{1}'+A_{2}'+A_{3}'.
}
It is clear that $|A_{1}'|+|A_{2}'|\lesssim I_{s_{0}}(u,v)^{2}\|w\|^{2}$.
Lemma \ref{freq.est.2} shows that
\EQQS{
A_{3}'=\LR{u(\HT\ds u)(J\ds^{2}+\HT\ds)w,\HT w}+\frac{1}{2}\LR{\ds(u\HT\ds u),(\HT u)^{2}},
}
which completes the proof.
\end{proof}


\begin{lem}\label{lem4.35.0}
Let $s_{0}>7/2$.
Let $u,v\in H^{s_{0}}(\T)$.
Then
\EQQS{
|\LR{(\HT\ds u)Jw,J\ds((\ds u)^{2}-(\ds v)^{2})}|
\lesssim I_{s_{0}}(u,v)^{2}\|w\|^{2},
}
where $w=u-v$.
\end{lem}

\begin{proof}
Set $z=u+v$.
Note that
\EQQS{
\LR{(\HT\ds u)Jw,J\ds((\ds u)^{2}-(\ds v)^{2})}
=-\LR{\ds(\ds zJ\ds((\HT\ds u)Jw)),w},
}
which shows the desired inequality.
\end{proof}


\begin{lem}\label{lem4.36.0}
Let $s_{0}>7/2$.
Let $u,v\in H^{s+2}(\T)$.
Then
\EQQS{
|\LR{(\HT\ds u)Jw,J\ds((\HT\ds u)^{2}-(\HT\ds v)^{2})}|
\lesssim I_{s_{0}}(u,v)^{2}\|w\|^{2},
}
where $w=u-v$.
\end{lem}

\begin{proof}
The proof is identical with that of the previous lemma.
\end{proof}

A similar argument to the proof of Lemma \ref{lem4.34.0} with using Corollary \ref{reduction3.H},
we can show the following:

\begin{lem}\label{lem4.37.0}
Let $s_{0}>7/2$.
Let $u,v\in H^{s_{0}}(\T)$.
Then
\EQQS{
|\LR{(\HT\ds u)Jw,\HT J\ds(u\HT\ds^{2}u-v\HT\ds^{2}v)}|
\lesssim I_{s_{0}}(u,v)^{2}\|w\|^{2},
}
where $w=u-v$.
\end{lem}

By the integration by parts with Lemma \ref{comm.est.H} and \ref{freq.est.2}, we obtain the following:

\begin{lem}\label{lem29}
Let $s_{0}>7/2$.
Let $u,v\in H^{s_{0}}(\T)$.
Then
\EQQS{
&\sum_{j=3}^{4}(|\LR{u\HT\ds(F_{j}(u)-F_{j}(v)),\HT Jw}|+|\LR{u\HT w,\HT J\ds(F_{j}(u)-F_{j}(v))}|)\\
&\lesssim I_{s_{0}}(u,v)^{4}\|w\|^{2},
}
where $w=u-v$.
\end{lem}

By the presence of $J$, we can easily obtain the following two lemmas:

\begin{lem}\label{lem30}
Let $s_{0}>7/2$.
Let $u,v\in H^{s_{0}}(\T)$.
Then
\EQQS{
\sum_{j=3}^{4}|\LR{(\HT\ds u)Jw,J\ds(F_{j}(u)-F_{j}(v))}|
\lesssim I_{s_{0}}(u,v)^{4}\|w\|^{2},
}
where $w=u-v$.
\end{lem}


\begin{lem}\label{lem31}
Let $s_{0}>7/2$.
Let $u,v\in H^{s_{0}}(\T)$.
Then
\EQQS{
\sum_{j=2}^{4}|\LR{u^{2}Jw,J\ds(F_{j}(u)-F_{j}(v))}|
\lesssim I_{s_{0}}(u,v)^{5}\|w\|^{2},
}
where $w=u-v$.
\end{lem}

\begin{proof}[Proof of Proposition \ref{ene.est.0}]
The proof is similar to that of Proposition \ref{ene.est.}.
Put $w:=u_{1}-u_{2}$.
Then $w$ satisfies \eqref{eq3.7} on $[0,T]$.
By Lemma \ref{NT0}, we have
\EQS{
\begin{aligned}\label{eq4.1}
&\left|\frac{1}{2}\frac{d}{dt}\|w\|^{2}+\la_{1}(0)\LR{\ds u_{1},(\ds w)^{2}}
 +\la_{2}(0)\LR{(\HT\ds^{2}u_{1})\HT\ds w,w}+\la_{3}(0)\LR{u_{1}\ds u_{1}\HT\ds w,w}\right|\\
&\le CI_{s_{0}}(u_{1},u_{2})^{3}\|w\|_{H^{s}}^{2}+\max\{\e_{1}^{2},\e_{2}^{2}\}\|u_{2}\|_{H^{s_{0}+1}}^{2}.
\end{aligned}
}
By Lemma \ref{m.e.1.2new0}, \ref{lem4.23.0}, \ref{lem4.25.0}, \ref{lem4.26.0}, \ref{lem4.28.0}, \ref{lem4.29.0}, \ref{lem4.30.0}, \ref{lem4.31.0}, \ref{lem4.32.0}, \ref{lem29}, we also have
\EQS{
\begin{aligned}\label{eq4.2}
&\left|\frac{d}{dt}M^{(1)}(u_{1},u_{2})-\la_{1}(0)\LR{\ds u_{1},(\ds w)^{2}}
 +\frac{\la_{1}(0)\la_{4}(0)}{4}\LR{(\HT\ds^{2}u_{1})\HT\ds w,w}\right|\\
&\le CI_{s_{0}}(u_{1},u_{2})^{5}\|w\|_{H^{s}}^{2}+\max\{\e_{1}^{2},\e_{2}^{2}\}\|u_{2}\|_{H^{s_{0}+1}}^{2}
\end{aligned}
}
Similarly, by Lemma \ref{lem2.14new0}, \ref{lem4.34.0}, \ref{lem4.35.0}, \ref{lem4.36.0}, \ref{lem4.37.0}, \ref{lem30}, we have
\EQS{
\begin{aligned}\label{eq4.3}
&\left|\frac{d}{dt}M^{(2)}(u_{1},u_{2})-\la_{2}(0)\LR{(\HT\ds^{2}u_{1})\HT\ds w,w}\right|\\
&\le CI_{s_{0}}(u_{1},u_{2})^{5}\|w\|_{H^{s}}^{2}+\max\{\e_{1}^{2},\e_{2}^{2}\}\|u_{2}\|_{H^{s_{0}+1}}^{2}.
\end{aligned}
}
Moreover, by Lemma \ref{lem2.21new0} and \ref{lem31}, we obtain
\EQS{
\begin{aligned}\label{eq4.4}
&\left|\frac{d}{dt}M^{(3)}(u_{1},u_{2})-\frac{\la_{1}(0)\la_{4}(0)+4\la_{3}(0)}{4}\LR{u\ds u\HT\ds w,w}\right|\\
&\le CI_{s_{0}}(u_{1},u_{2})^{5}\|w\|_{H^{s}}^{2}+\max\{\e_{1}^{2},\e_{2}^{2}\}\|u_{2}\|_{H^{s_{0}+1}}^{2}.
\end{aligned}
}
It is easy to see that
\EQS{
\begin{aligned}\label{eq4.5}
\frac{d}{dt}\{\|w\|_{H^{-1}}^{2}(1+C\|u_{1}\|^{2}+C\|u_{1}\|^{4})\}\le CI_{s_{0}}(u_{1},u_{2})^{7}\|w\|^{2}.
\end{aligned}
}
Indeed, we have
\EQQS{
&\LR{\LR{D}^{-1}\ds(u_{1}\ds^{2}u_{1}-u_{2}\ds^{2}u_{2}),\LR{D}^{-1}w}\\
&=-\LR{[\LR{D}^{-1}\ds^{2},u_{1}]w,\LR{D}^{-1}\ds w}-\frac{1}{2}\LR{\ds u_{1},(\LR{D}^{-1}\ds w)^{2}}\\
&\quad+2\LR{\LR{D}^{-1}\ds(\ds u_{1}w),\LR{D}^{-1}\ds w}-\LR{\LR{D}^{-1}(\ds^{2}u_{1}w),\LR{D}^{-1}\ds w}\\
&\quad+\LR{\LR{D}^{-1}\ds(w\ds^{2}u_{2}),\LR{D}^{-1}w},
}
which together with Lemma \ref{comm.est.4.1} implies that
\[|\LR{\LR{D}^{-1}\ds(u_{1}\ds^{2}u_{1}-u_{2}\ds^{2}u_{2}),\LR{D}^{-1}w}|\lesssim I_{s_{0}}(u_{1},u_{2})\|w\|^{2}.\]
Other terms can be estimated in a simiar way, and then we obtain \eqref{eq4.5}.
Therefore, collecting \eqref{eq4.1}, \eqref{eq4.2}, \eqref{eq4.3}, \eqref{eq4.4} and \eqref{eq4.5}, we obtain \eqref{eq2.5}.
\end{proof}

\section*{Acknowlegdements}
The author would like to express his deep gratitude to Professor Kotaro Tsugawa for encouragement and valuable comments, especially for Definition \ref{J}.

\end{document}